\documentclass{amsart}
\usepackage{hyperref,amssymb,pifont}
\usepackage{pstricks, pst-node,pst-tree,multido}
\usepackage{a4wide}

\title{Proofs of two conjectures of Kenyon and Wilson on Dyck tilings}
\author{Jang Soo Kim}
\email{kimjs@math.umn.edu}
\keywords{Dyck paths, Dyck tilings}
\date{\today}
\newtheorem{thm}{Theorem}[section]
\newtheorem{lem}[thm]{Lemma}
\newtheorem{prop}[thm]{Proposition}
\newtheorem{cor}[thm]{Corollary}
\theoremstyle{definition}

\newtheorem{defn}{Definition}

\theoremstyle{remark}

\newcommand\D{\mathcal{D}}
\newcommand\TD{\mathcal{TD}}
\newcommand\C{\mathrm{Chord}}

\newcommand\M{\mathcal{M}}
\newcommand\HH{\mathcal{H}}

\newcommand\area{\mathrm{area}}

\newcommand\cro{\mathrm{cr}}
\newcommand\nest{\mathrm{ne}}

\newcommand\Dyck{\mathrm{Dyck}}

\def\sch.{Schr{\"o}der}

\newcommand\HT{\mathrm{ht}}
\newcommand\lm{\lambda/\mu}

\newcommand\qint[1]{\left[ #1\right]_q}
\newcommand\pqint[1]{\left[ #1\right]_{p,q}}
\newcommand\Qbinom[3]{\genfrac{[}{]}{0pt}{}{#1}{#2}_{#3}}
\newcommand\qbinom[2]{\Qbinom{#1}{#2}{q}}

\newcommand{\qhyper}[5]{{_{#1}\phi_{#2}}\left[\genfrac{}{}{0pt}{}{#3}{#4};q,#5\right]}

\psset{gridlabels=0pt,subgriddiv=1, gridwidth=.1pt,gridcolor=gray}
\psset{linewidth=1.5pt, unit=.6cm}
\psset{unit=12pt,dimen=middle}

\newcount\ax \newcount\ay
\newcount\bx \newcount\by
\newcount\cx \newcount\cy
\newcount\dx \newcount\dy
\def\cell(#1,#2)[#3]{
\ax=#2 \ay=#1
\multiply\ay by-1
\bx=\ax \by=\ay
\cx=\ax \cy=\ay
\dx=\ax \dy=\ay
\advance\bx by-1
\advance\dy by1
\advance\cx by-1
\advance\cy by1
\psline (\dx,\dy)(\ax,\ay)(\bx,\by)
\rput(\number\cx.5,
\ifnum\cy=0 -0.5\else\number\cy.5\fi){#3}
}

\def\dia(#1,#2){
  \rput(#1,#2){\pspolygon(1,0)(0,1)(-1,0)(0,-1)}
  }

\def\dyckgrid#1{
\psgrid(0,0)(#1,#1)
\psline[linewidth=.1pt, linecolor=gray](0,0)(#1,#1)
}

\def\HS(#1,#2){\rput(#1,#2){\psline[linewidth=1pt](0,0)(1,0)}}
\def\HSC(#1,#2){\rput(#1,#2){   
\pscurve [linewidth=1pt] (0,0)(.2,.1)(.4,-.1)(.6,.1)(.8,-.1)(1,0)
}}
\def\UP(#1,#2){\rput(#1,#2){\psline[linewidth=1pt](0,0)(1,1) 
}}
\def\DW(#1,#2){\rput(#1,#2){\psline[linewidth=1pt](0,0)(1,-1) 
}}
\def\MUP(#1,#2){\rput(#1,#2){\psline[linecolor=red, linewidth=2pt](0,0)(1,1)
}}
\def\MDW(#1,#2){\rput(#1,#2){\psline[linecolor=red, linewidth=2pt](0,0)(1,-1)
}}
\def\MHS(#1,#2){\rput(#1,#2){\psline[linecolor=red, linewidth=2pt](0,0)(1,0)}}
\def\MHSC(#1,#2){\rput(#1,#2){   
\pscurve [linecolor=red, linewidth=2pt] (0,0)(.2,.1)(.4,-.1)(.6,.1)(.8,-.1)(1,0)
}}

\def\cvput#1[#2]{\pnode(#1,1){#1} \pscircle*(#1,1){.1} \rput(#1,.5){$#2$}}
\def\vput#1{\cvput#1[#1]}
\def\edge#1#2{\ncarc[arcangle=50]{#1}{#2}}
\begin{document}

\begin{abstract}
  Recently, Kenyon and Wilson introduced a certain matrix $M$ in order to compute
  pairing probabilities of what they call the double-dimer model. They showed
  that the absolute value of each entry of the inverse matrix $M^{-1}$ is equal
  to the number of certain Dyck tilings of a skew shape. They conjectured two
  formulas on the sum of the absolute values of the entries in a row or a column
  of $M^{-1}$.  In this paper we prove the two conjectures. As a consequence we
  obtain that the sum of the absolute values of all entries of $M^{-1}$ is equal
  to the number of complete matchings. We also find a bijection between Dyck
  tilings and complete matchings.
\end{abstract}

\maketitle
% \tableofcontents

\section{Introduction}
\label{sec:introduction}

A \emph{Dyck path} of length $2n$ is a lattice path from $(0,0)$ to $(n,n)$
consisting of \emph{up steps} $(0,1)$ and \emph{down steps} $(1,0)$ that never
goes below the line $y=x$. The set of Dyck paths of length $2n$ is denoted
$\Dyck(2n)$. In this paper we will sometimes identify a Dyck path $\lambda$ with
a partition as shown in Figure~\ref{fig:partition}. For $\lambda,\mu\in
\Dyck(2n)$, if $\mu$ is above $\lambda$, then the skew shape $\lm$ is well
defined.

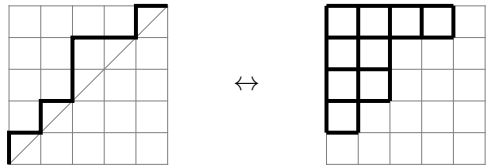
\begin{figure}
  \centering
\begin{pspicture}(0,0)(5,5)
\dyckgrid5
\psline(0,0)(0,1)(1,1)(1,2)(2,2)(2,4)(4,4)(4,5)(5,5)
\end{pspicture}%
\begin{pspicture}(0,0)(5,5)
\rput(2.5,2.5){$\leftrightarrow$}
\end{pspicture}%
\begin{pspicture}(0,0)(5,5)
\psgrid (0,0)(5,5) 
\psline(0,5)(4,5)
\psline(0,4)(4,4)
\psline(0,3)(2,3)
\psline(0,2)(2,2)
\psline(0,1)(1,1)
\psline(0,1)(0,5)
\psline(1,1)(1,5)
\psline(2,2)(2,5)
\psline(3,4)(3,5)
\psline(4,4)(4,5)
\end{pspicture}
\caption{A Dyck path identified with the partition $(4,2,2,1)$.}
  \label{fig:partition}
\end{figure}

For two Dyck paths $\lambda$ and $\mu$, we define $\lambda\succ\mu$ if $\lambda$
can be obtained from $\mu$ by choosing some matching pairs of up steps and down
steps and exchanging the chosen up steps and down steps, see
Figure~\ref{fig:order}. In order to compute pairing probabilities of so-called
the double-dimer model, Kenyon and Wilson \cite{Kenyon2011a, Kenyon2011}
introduced a matrix $M$ defined as follows.  The rows and columns of $M$ are
indexed by $\lambda,\mu\in \Dyck(2n)$, and $M_{\lambda,\mu} =1$ if
$\lambda\succ\mu$, and $M_{\lambda,\mu} =0$ otherwise.

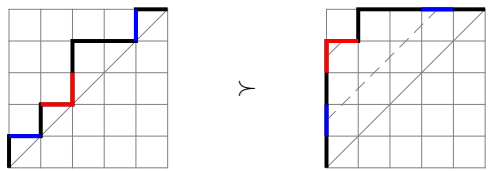
\begin{figure}
  \centering
\begin{pspicture}(0,0)(5,5)
\dyckgrid5
\psline(0,0)(0,1)(1,1)(1,2)(2,2)(2,4)(4,4)(4,5)(5,5)
{\psset{linecolor=blue} \psline(0,1)(1,1) \psline(4,4)(4,5)} 
{\psset{linecolor=red} \psline(1,2)(2,2)(2,3)}
\end{pspicture}%
\begin{pspicture}(0,0)(5,5)
\rput(2.5,2.5){$\succ$}
\end{pspicture}%
\begin{pspicture}(0,0)(5,5)
\dyckgrid5
\psline(0,0)(0,4)(1,4)(1,5)(5,5)
{
\psset{linewidth=.1pt,linestyle=dashed, linecolor=gray}
\psline[linestyle=dashed](0,1.5)(3.5,5)
\psline[linestyle=dashed](0,3.5)(.5,4)
}
{\psset{linecolor=blue} \psline(0,1)(0,2) \psline(3,5)(4,5)} 
{\psset{linecolor=red} \psline(0,3)(0,4)(1,4)} 
\end{pspicture}%
\caption{An example of the order $\succ$ on Dyck paths.}
  \label{fig:order}
\end{figure}

A \emph{ribbon} is a connected skew shape which does not contain a $2\times 2$
box.  A \emph{Dyck tile} is a ribbon such that the centers of the cells form a
Dyck path. The \emph{length} of a Dyck tile is the length of the Dyck path
obtained by joining the centers of the cells, see Figure~\ref{fig:dycktile}.

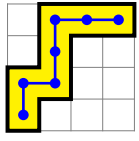
\begin{figure}
  \centering
\begin{pspicture}(0,0)(4,4)
  \psgrid (0,0)(4,4)  
\pspolygon[fillstyle=solid,fillcolor=yellow] (0,0)(0,2)(1,2)(1,4)(4,4)(4,3)(2,3)(2,1)(1,1)(1,0)(0,0)
\psline[showpoints=true,linewidth=1pt,linecolor=blue](.5,.5)(.5,1.5)(1.5,1.5)(1.5,2.5)(1.5,3.5)(2.5,3.5)(3.5,3.5)
\end{pspicture}%
 \caption{A Dyck tile of length $6$.}
  \label{fig:dycktile}
\end{figure}

For $\lambda,\mu\in\Dyck(2n)$, a \emph{(cover-inclusive) Dyck tiling} of $\lm$
is a tiling with Dyck tiles such that for two Dyck tiles $T_1$ and $T_2$, if
$T_1$ has a cell to the southeast of a cell of $T_2$, then a southeast
translation of $T_2$ is contained in $T_1$. We denote by $\D(\lm)$ the set of
Dyck tilings of $\lm$. For $T\in\D(\lm)$, we denote by $|T|$ the number of tiles
in $T$.  Note that $\D(\lambda/\lambda)$ has only one tiling, the empty tiling
$\emptyset$ with $|\emptyset|=0$. 

Kenyon and Wilson \cite[Theorem~1.5]{Kenyon2011} showed that the inverse
matrix $M^{-1}$ of $M$ can be expressed using Dyck tilings:
\[
  M^{-1}_{\lambda,\mu}  = (-1)^{|\lm|} \times |\D(\lm)|,
\]
where $|\lm|$ denotes the number of cells in $\lm$. 

In this paper we prove two conjectures of Kenyon and Wilson on $q$-analogs of
the sum of the absolute values of the entries in a row or a column of $M^{-1}$. In
order to state the conjectures we need the following notions.

A \emph{chord} of a Dyck path $\lambda$ is a matching pair of an up step and a
down step. We denote by $\C(\lambda)$ the set of chords of $\lambda$. For
$c\in\C(\lambda)$, the \emph{length} $|c|$ of $c$ is defined to be the
difference between the $x$-coordinates of the starting point of the up step and
the ending point of the down step.  The \emph{height} $\HT(c)$ of $c$ is defined
to be $i$ if $c$ is between the lines $y=x+i-1$ and $y=x+i$. See
Figure~\ref{fig:chords} for an example.

We use the standard notations for $q$-integers:
\begin{align*}
\qint{n} &= 1+q+\cdots +q^{n-1} , \\
\qint{n}! &= \qint{1} \qint{2} \cdots \qint{n}.
\end{align*}
Also, we denote $[n]=\{1,2,\dots,n\}$, which should not be confused with 
the $q$-integers. 

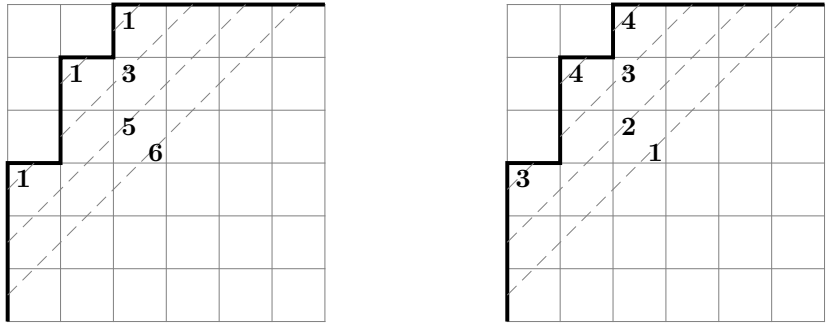
\begin{figure}
  \centering
\psset{unit=20pt}
\bf
  \begin{pspicture}(0,0)(6,6)
   \psgrid (0,0)(6,6) 
\psline(0,0)(0,3)(1,3)(1,5)(2,5)(2,6)(6,6)
\psset{linewidth=.1pt,linestyle=dashed, linecolor=gray}
\psline(0,.5)(5.5,6)
\psline(0,1.5)(4.5,6)
\psline(0,2.5)(.5,3)
\psline(1,3.5)(3.5,6)
\psline(1,4.5)(1.5,5)
\psline(2,5.5)(2.5,6)
\rput(-.2,.2){
\rput(3,3){6}
\rput(2.5,3.5){5}
\rput(.5,2.5){1}
\rput(2.5,4.5){3}
\rput(1.5,4.5){1}
\rput(2.5,5.5){1}
}
  \end{pspicture}\qquad \qquad \qquad
  \begin{pspicture}(0,0)(6,6)
   \psgrid (0,0)(6,6) 
\psline(0,0)(0,3)(1,3)(1,5)(2,5)(2,6)(6,6)
\psset{linewidth=.1pt,linestyle=dashed, linecolor=gray}
\psline(0,.5)(5.5,6)
\psline(0,1.5)(4.5,6)
\psline(0,2.5)(.5,3)
\psline(1,3.5)(3.5,6)
\psline(1,4.5)(1.5,5)
\psline(2,5.5)(2.5,6)
\rput(-.2,.2){
\rput(3,3){1}
\rput(2.5,3.5){2}
\rput(.5,2.5){3}
\rput(2.5,4.5){3}
\rput(1.5,4.5){4}
\rput(2.5,5.5){4}
}
  \end{pspicture}
\caption{The lengths (left) and the heights (right) of the chords of a Dyck path.}
  \label{fig:chords}
\end{figure}

We now state the main theorems. 

\begin{thm}\cite[Conjecture 1]{Kenyon2011}\label{thm:1}
  Given a Dyck path $\lambda\in\Dyck(2n)$, we have
\[
\sum_{\mu\in\Dyck(2n)} \sum_{T\in\D(\lm)} q^{(|\lm|+|T|)/2} 
= \frac{\qint{n}!}{\prod_{c\in \C(\lambda)} \qint{|c|}}.
\]
\end{thm}

\begin{thm}\cite[Conjecture 2]{Kenyon2011}\label{thm:2}
  Given a Dyck path $\mu\in\Dyck(2n)$, we have
\[
\sum_{\lambda\in\Dyck(2n)} \sum_{T\in\D(\lm)} q^{|T|} 
= \prod_{c\in \C(\mu)} \qint{\HT(c)}.
\]
\end{thm}

Our proof of Theorem~\ref{thm:2} is simpler than the proof of
Theorem~\ref{thm:1}. So we will first present the proof of Theorem~\ref{thm:2}.

The rest of the paper is organized as follows. In Section~\ref{sec:proof2} we
prove Theorem~\ref{thm:2}.  In Section~\ref{sec:trunc-dyck-tilings} we introduce
truncated Dyck tilings, which are very similar to Dyck tilings, and some
properties of them. We then state a generalization of Theorem~\ref{thm:1}. In
Section~\ref{sec:proof1} we prove the generalization of Theorem~\ref{thm:1}.  In
Section~\ref{sec:anoth-proof-prop} we give another proof of an important
identity used in the proof of the generalization of Theorem~\ref{thm:1}.  In
Section~\ref{sec:bijection} we construct a bijection between Dyck tilings and
complete matchings, and discuss some applications of the bijection.  In
Section~\ref{sec:final-remarks} we give final remarks. 

\section{Proof of Theorem~\ref{thm:2}}
\label{sec:proof2}

We denote by $\delta_{n-1}$ the staircase partition $(n-1,n-2,\dots,1)$. Note
that if $\mu\in\Dyck(2n)$, we have $\mu\subseteq\delta_{n-1}$.

We will prove Theorem~\ref{thm:2} by induction on the number $m(\mu)$ of cells
in $|\delta_{n-1}/\mu|$, where $n$ is the half-length of $\mu$. If $m(\mu)=0$,
then $\mu=\delta_{n-1}$ and the theorem is clear.  Assume $m\geq1$ and the
theorem is true for all $\nu$ with $m(\nu)<m$.  Now suppose $\mu\in\Dyck(2n)$
with $m(\mu)=m$.  Since $|\delta_{n-1}/\mu|=m\geq1$, we can pick a cell $s\in
\delta_{n-1}/\mu$ such that $\mu\cup s$ is also a partition. Let $h$ be the
height of the chord of $\mu$ contained in $s$.  Consider a tiling $T\in\D(\lm)$
for some $\lambda\in\Dyck(2n)$.  Then there are two cases.

Case 1: $s$ by itself is a tile in $T$.  Let $\mu'=\mu\cup\{s\}$.  Then
$T'=T\setminus\{s\}$ is a tiling in $\D(\lambda/\mu')$. Thus the sum of
$q^{|T|}$ for all such choices of $\lambda$ and $T$ equals
\begin{align*}
\sum_{\lambda\in \Dyck(2n)}
\sum_{T'\in\D(\lambda/\mu')} q^{|T'|+1}.
\end{align*}
By the induction hypothesis, the above sum is equal to
\begin{equation}
  \label{eq:15}
q \prod_{c\in \C(\mu')} \qint{\HT(c)}
=\frac{q\qint{h-1}}{\qint{h}} \prod_{c\in \C(\mu)} \qint{\HT(c)},
\end{equation}
where we use the fact that $\mu'$ has one more chord of height $h-1$ and one
less chord of height $h$ than $\mu$, see Figure~\ref{fig:notcovered}.

\begin{figure}
  \centering
\psset{unit=20pt}
  \begin{pspicture}(0,0)(6,6)
    \dyckgrid6
   \psline(0,0)(0,2)(1,2)(1,4)(2,4)(2,5)(3,5)(3,6)(6,6)
\psset{linewidth=.1pt,linestyle=dashed, linecolor=gray}
\psline(2,4.5)(2.5,5)
\psline(1,2.5)(4.5,6)
\rput(-.2,.2){
\rput(3.5,4){$h-1$}
\rput(2.5,4.5){$h$}
}
  \end{pspicture}\qquad\qquad
  \begin{pspicture}(0,0)(6,6)
    \dyckgrid6
   \psline(0,0)(0,2)(1,2)(1,4)(3,4)(3,6)(6,6)
\psframe[linecolor=blue](2,4)(3,5)
\rput(2.5,4.5){$s$}
\psset{linewidth=.1pt,linestyle=dashed, linecolor=gray}
\psline(1,2.5)(2.5,4)
\psline(3,4.5)(4.5,6)
\rput(-.2,.2){
\rput(2.5,3){$h-1$}
\rput(4.5,5){$h-1$}
}
  \end{pspicture}
  \caption{$\mu\cup\{s\}$ has one more chord of height $h-1$ and one less chord
    of height $h$ than $\mu$.}
  \label{fig:notcovered}
\end{figure}
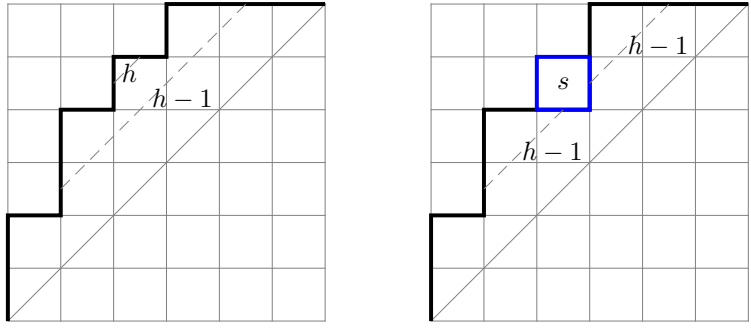

Case 2: Otherwise we have either $s\not\in \lm$ or $s$ is covered by a Dyck tile
of length greater than $0$. Then we collapse the slice containing $s$, in other
words, remove the region in $\lm$ bounded by the two lines with slope $-1$
passing through the northeast corner and the southwest corner of $s$ and attach
the two remaining regions, see Figure~\ref{fig:slice}. Let $\lambda'$, $\mu'$,
and $T'$ be the resulting objects obtained from $\lambda$, $\mu$, and $T$ in
this way. Since the collapsed slice is completely covered by Dyck tiles of
length greater than $0$, the original objects $\lambda$, $\mu$, and $T$ can be
obtained from $\lambda'$, $\mu'$, and $T'$. We also have $|T|=|T'|$.  Thus, by
the induction hypothesis, the sum of $q^{|T|}$ for all possible choices
of $\lambda$ and $T$ is equal to
\begin{equation}
  \label{eq:24}
\sum_{\lambda'\in \Dyck(2n)} 
\sum_{T'\in\D(\lambda'/\mu')} q^{|T'|}
=\prod_{c\in \C(\mu')} \qint{\HT(c)}=\frac{1}{\qint{h}} \prod_{c\in \C(\mu)} \qint{\HT(c)}.
\end{equation}

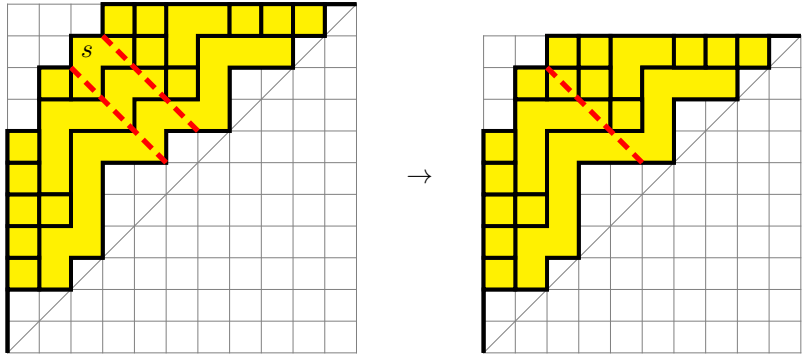
\begin{figure}
  \centering
  \begin{pspicture}(0,0)(11,11) 
\dyckgrid{11}
\psline(0,0)(0,7)(1,7)(1,9)(2,9)(2,10)(3,10)(3,11)(11,11)
\psline(0,0)(0,2)(2,2)(2,3)(3,3)(3,6)(5,6)(5,7)(7,7)(7,9)(9,9)(9,10)(10,10)(10,11)(11,11)
\psset{fillstyle=solid,fillcolor=yellow}
\rput(-1,1){
\pspolygon(2,4)(2,5)(2,7)(4,7)(4,8)(5,8)(6,8)(6,7)(5,6)(3,6)(3,4)
\rput(-1,1){
\pspolygon(10,8)(8,8)(8,6)(6,6)(6,5)(4,5)(4,2)(3,2)(3,1)(3,0)(4,0)(4,1)(5,1)(5,4)(7,4)(7,5)(9,5)(9,7)(11,7)(11,8)(10,8)}
\pspolygon(6,10)(7,10)(8,10)(8,9)(7,9)(7,8)(6,8)(6,9)(6,10)
\rput(-3,-1){\pspolygon(6,10)(7,10)(8,10)(8,9)(7,9)(7,8)(6,8)(6,9)(6,10)}
}
\psframe(0,2)(1,3)
\rput(0,1){\psframe(0,2)(1,3)}
\rput(0,2){\psframe(0,2)(1,3)}
\rput(0,3){\psframe(0,2)(1,3)}
\rput(0,4){\psframe(0,2)(1,3)}
\rput(1,2){\psframe(0,2)(1,3)}
\rput(1,6){\psframe(0,2)(1,3)}
\rput(5,6){\psframe(0,2)(1,3)}
\rput(4,7){\psframe(0,2)(1,3)}
\rput(4,8){\psframe(0,2)(1,3)}
\rput(3,8){\psframe(0,2)(1,3)}
\rput(7,8){\psframe(0,2)(1,3)}
\rput(8,8){\psframe(0,2)(1,3)}
\rput(9,8){\psframe(0,2)(1,3)}
\psline[linewidth=2pt, linestyle=dashed,linecolor=red](2,9)(5,6)
\psline[linewidth=2pt, linestyle=dashed,linecolor=red](3,10)(6,7)
\rput(2.5,9.5){$s$}
\end{pspicture}%
  \begin{pspicture}(0,0)(4,11)
\rput(2,5.5){$\rightarrow$}
\end{pspicture}%
  \begin{pspicture}(0,0)(11,11)
\dyckgrid{10}
\psline(0,0)(0,7)(1,7)(1,9)(2,9)(2,10)(10,10)
\psline(0,0)(0,2)(2,2)(2,3)(3,3)(3,6)(5,6)(6,6)(6,8)(8,8)(8,9)(9,9)(9,10)(10,10)
\psset{fillstyle=solid,fillcolor=yellow}
\rput(-1,1){
\pspolygon(3,4)(2,4)(2,7)(5,7)(5,6)(3,6)(3,4)
\rput(-1,-1){\pspolygon(6,10)(8,10)(8,9)(7,9)(7,8)(6,8)(6,10)}
}
\pspolygon (2,2)(2,3)(3,3)(3,6)(5,6)(6,6)(6,8)(8,8)(8,9)(5,9)(5,7)(2,7)(2,4)(1,4)(1,2)
\rput(0,0){\psframe(0,2)(1,3)}
\rput(0,1){\psframe(0,2)(1,3)}
\rput(0,2){\psframe(0,2)(1,3)}
\rput(0,3){\psframe(0,2)(1,3)}
\rput(0,4){\psframe(0,2)(1,3)}
\rput(1,2){\psframe(0,2)(1,3)}
\rput(1,6){\psframe(0,2)(1,3)}
\rput(2,6){\psframe(0,2)(1,3)}
\rput(2,7){\psframe(0,2)(1,3)}
\rput(3,6){\psframe(0,2)(1,3)}
\rput(3,7){\psframe(0,2)(1,3)}
\rput(4,5){\psframe(0,2)(1,3)}
\rput(6,7){\psframe(0,2)(1,3)}
\rput(7,7){\psframe(0,2)(1,3)}
\rput(8,7){\psframe(0,2)(1,3)}
\psline[linewidth=2pt,linestyle=dashed,linecolor=red](2,9)(5,6)
\end{pspicture}
 \caption{Collapsing the slice containing $s$.}
  \label{fig:slice}
\end{figure}

Summing \eqref{eq:15} and \eqref{eq:24}, the theorem is also true for $\mu$.  By
induction, the theorem is proved.

We note that the proof in this section was also discovered independently by
Matja\v z Konvalinka (personal communication with Matja\v z Konvalinka).

It is not difficult to construct a bijection between Dyck tilings and Hermite
histories (see Section~\ref{sec:bijection} for the definition) by the same
recursive manner as in the proof in this section. In fact, the bijection
obtained in this way has a non-recursive description, which we will present in
Section~\ref{sec:bijection}.

\section{Truncated Dyck tilings}
\label{sec:trunc-dyck-tilings}

In this section we state a generalization of Theorem~\ref{thm:1}. We first need
to reformulate Theorem~\ref{thm:1}. 

For a Dyck tiling $T$ we define $\|T\|$ to be the sum of the half-lengths of all
Dyck tiles in $T$.

\begin{lem}\label{lem:reform}
For $T\in \D(\lm)$, we have
\[
q^{(|\lm|+|T|)/2}   = q^{|\lm|-\|T\|}.
\]
\end{lem}
\begin{proof}
  Let $\eta$ be a Dyck tile in $T$. We will compute the contribution of $\eta$
  as a factor in both sides of the equation. Suppose $\eta$ is of length
  $2k$. Then $|\eta|=2k+1$, and the contribution of $\eta$ in the left hand side
  (resp.~right hand side) is $q^{((2k+1)+1)/2}=q^{k+1}$
  (resp.~$q^{(2k+1)-k}=q^{k+1}$). Since each tile contributes the same factor in
  both sides we get the equation.
\end{proof}

By Lemma~\ref{lem:reform}, we can rewrite Theorem~\ref{thm:1} as follows.

\begin{thm}\label{thm:1r}
  Given a Dyck path $\lambda\in\Dyck(2n)$, we have
\[
\sum_{\mu\in\Dyck(2n)} \sum_{T\in\D(\lm)} q^{|\lm|-\|T\|} 
= \frac{\qint{n}!}{\prod_{c\in \C(\lambda)} \qint{|c|}}.
\]
\end{thm}

\begin{figure}
  \centering
  \begin{pspicture}(0,-1)(3,4)
\dyckgrid3
\psline[linecolor=blue](0,0)(0,2)(1,2)(1,3)(3,3)
  \end{pspicture}%
  \begin{pspicture}(0,0)(2,5)
\rput(1,2.5){$*$}
  \end{pspicture}%
 \begin{pspicture}(0,-1.5)(2,3.5)
\dyckgrid2
\psline[linecolor=red](0,0)(0,2)(2,2)
  \end{pspicture}%
  \begin{pspicture}(0,0)(2,5)
\rput(1,2.5){$=$}
  \end{pspicture}%
 \begin{pspicture}(0,0)(5,5)
\dyckgrid5
\psline[linecolor=blue](0,0)(0,2)(1,2)(1,3)(3,3)
\rput(3,3){\psline[linecolor=red](0,0)(0,2)(2,2)}
\end{pspicture}
 \caption{The $*$ operation on two Dyck paths.}
  \label{fig:operation}
\end{figure}
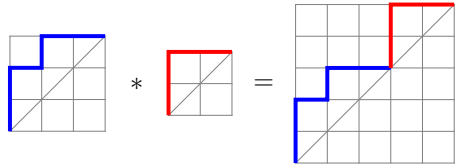

For two Dyck paths $\lambda$ and $\mu$ (not necessarily of the same length) we
define $\lambda*\mu$ to be the Dyck path obtained from $\lambda$ by attaching
$\mu$ at the end of $\lambda$, see Figure~\ref{fig:operation}.  For a
nonnegative integer $k$, we denote by $\Delta_k$ the Dyck path of length $2k$
consisting of $k$ consecutive up steps and $k$ consecutive down steps.  For
nonnegative integers $k_1,\dots,k_r$, we define
\[
\Delta_{k_1,\dots,k_r} = \Delta_{k_1} * \cdots * \Delta_{k_r}.
\]

For an object $X$, which may be a point, a lattice path, or a tile, we denote by
$X+(i,j)$ the translation of $X$ by $(i,j)$.  So far, we have only considered
$\lm$ for two Dyck paths $\lambda$ and $\mu$ starting and ending at the same
points. We extend this definition as follows.

Suppose $\lambda$ is a Dyck path from $O=(0,0)$ to $N=(n,n)$ and $\mu$ is a
lattice path from $P=O+(-a,a)$ to $Q=N+(-b,b)$ for some nonnegative integers $a$
and $b$ such that $\mu$ never goes below $\lambda$. Then we define $\lm$ to be
the region bounded by $\lambda$, $\mu$, and the segments $OP$ and $NQ$.  We
denote by $|\lm|$ the area of the region $\lm$. Note that this notation is
consistent with the number $|\lm|$ of cells of $\lm$ when $\lm$ is a skew
shape. Given $\lambda$, $a$, and $b$, we denote by $L(\lambda; a,b)$ the set of
all lattice paths from $P$ to $Q$ which never go below $\lambda$.

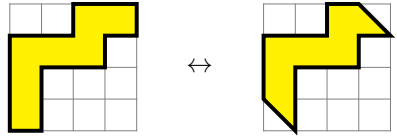
\begin{figure}
  \centering
  \begin{pspicture}(0,0)(4,4)
\psgrid(0,0)(4,4)
\psset{fillstyle=solid,fillcolor=yellow}
\rput(-2,-4){
\pspolygon(2,4)(2,5)(2,7)(4,7)(4,8)(5,8)(6,8)(6,7)(5,7)(5,6)(3,6)(3,4)
}
\end{pspicture}%
 \begin{pspicture}(0,0)(4,4)
\rput(2,2){$\leftrightarrow$}
\end{pspicture}%
  \begin{pspicture}(0,0)(4,4)
\psgrid(0,0)(4,4)
\psset{fillstyle=solid,fillcolor=yellow}
\rput(-2,-4){
\pspolygon(2,5)(2,7)(4,7)(4,8)(5,8)(6,7)(5,7)(5,6)(3,6)(3,4)
}
\end{pspicture}%
 \caption{A Dyck tile and the corresponding truncated Dyck tile.}
  \label{fig:trucated_Dyck_tile}
\end{figure}

\begin{defn}\label{defn:TD}
  A \emph{truncated Dyck tile} is a tile obtained from a Dyck tile of positive
  length by cutting off the northeast half-cell and the southwest half-cell as
  shown in Figure~\ref{fig:trucated_Dyck_tile}. A \emph{(cover-inclusive)
    truncated Dyck tiling} of a region $\lm$ is a tiling $T$ of a sub-region of
  $\lm$ with truncated Dyck tiles satisfying the following conditions:
\begin{itemize}
\item For each tile $\eta\in T$, if $(\eta+(1,-1))\cap \lm\ne\emptyset$, then
  there is another tile $\eta'\in T$ containing $(\eta+(1,-1))$. 
\item There are no two tiles sharing a border with slope $-1$.
\end{itemize}
Let $\TD(\lm)$ denote the set of truncated Dyck tilings of $\lm$. 
\end{defn}

If $\mu\in L(\lambda;0,0)$, there is a natural bijection between $\D(\lm)$ and
$\TD(\lm)$ as follows. For every tile in $T\in \D(\lm)$, remove the northeast
half-cell and the southwest half-cell as shown in
Figure~\ref{fig:dyck-trunc}. Note that the Dyck tiles of length $0$ simply
disappear.

\begin{figure}
  \centering
  \begin{pspicture}(0,0)(11,11) 
\dyckgrid{11}
\psline(0,0)(0,7)(1,7)(1,9)(2,9)(2,10)(3,10)(3,11)(11,11)
\psline(0,0)(0,2)(2,2)(2,3)(3,3)(3,6)(5,6)(5,7)(7,7)(7,9)(9,9)(9,10)(10,10)(10,11)(11,11)
\psset{fillstyle=solid,fillcolor=yellow}
\rput(-1,1){
\pspolygon(2,4)(2,5)(2,7)(4,7)(4,8)(5,8)(6,8)(6,7)(5,6)(3,6)(3,4)
\rput(-1,1){
\pspolygon(10,8)(8,8)(8,6)(6,6)(6,5)(4,5)(4,2)(3,2)(3,1)(3,0)(4,0)(4,1)(5,1)(5,4)(7,4)(7,5)(9,5)(9,7)(11,7)(11,8)(10,8)}
\pspolygon(6,10)(7,10)(8,10)(8,9)(7,9)(7,8)(6,8)(6,9)(6,10)
\rput(-3,-1){\pspolygon(6,10)(7,10)(8,10)(8,9)(7,9)(7,8)(6,8)(6,9)(6,10)}
}
\psframe(0,2)(1,3)
\rput(0,1){\psframe(0,2)(1,3)}
\rput(0,2){\psframe(0,2)(1,3)}
\rput(0,3){\psframe(0,2)(1,3)}
\rput(0,4){\psframe(0,2)(1,3)}
\rput(1,2){\psframe(0,2)(1,3)}
\rput(1,6){\psframe(0,2)(1,3)}
\rput(5,6){\psframe(0,2)(1,3)}
\rput(4,7){\psframe(0,2)(1,3)}
\rput(4,8){\psframe(0,2)(1,3)}
\rput(3,8){\psframe(0,2)(1,3)}
\rput(7,8){\psframe(0,2)(1,3)}
\rput(8,8){\psframe(0,2)(1,3)}
\rput(9,8){\psframe(0,2)(1,3)}
\end{pspicture}%
 \begin{pspicture}(0,0)(4,11)
\rput(2,5.5){$\leftrightarrow$}
\end{pspicture}%
  \begin{pspicture}(0,0)(11,11) 
\dyckgrid{11}
\psline(0,0)(0,7)(1,7)(1,9)(2,9)(2,10)(3,10)(3,11)(11,11)
\psline(0,0)(0,2)(2,2)(2,3)(3,3)(3,6)(5,6)(5,7)(7,7)(7,9)(9,9)(9,10)(10,10)(10,11)(11,11)
\psset{fillstyle=solid,fillcolor=yellow}
\rput(-1,1){
\pspolygon(3,4)(2,5)(2,7)(4,7)(4,8)(5,8)(6,7)(5,6)(3,6)(3,4)
\rput(-1,1){
\pspolygon(8,8)(8,6)(6,6)(6,5)(4,5)(4,2)(3,2)(3,1)(4,0)(4,1)(5,1)(5,4)(7,4)(7,5)(9,5)(9,7)(11,7)(10,8)}
\pspolygon(6,10)(7,10)(8,9)(7,9)(7,8)(6,9)(6,10)
\rput(-3,-1){\pspolygon(6,10)(7,10)(8,9)(7,9)(7,8)(6,9)(6,10)}
}
\end{pspicture}%
\caption{A Dyck tiling and the corresponding truncated Dyck tiling.}
  \label{fig:dyck-trunc}
\end{figure}
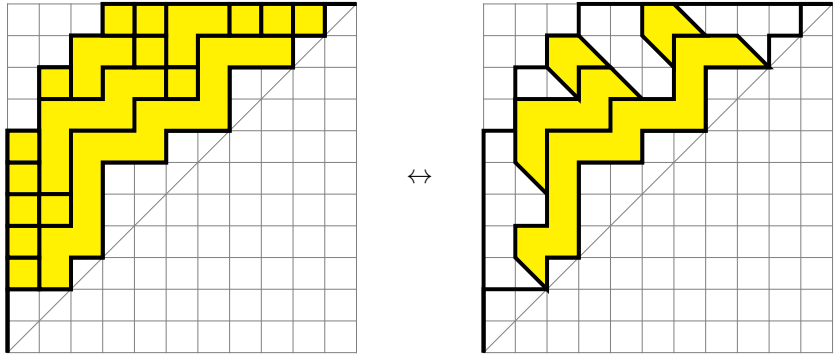

Let
\[
B_q(\lambda;a,b) = \sum_{\mu\in L(\lambda;a,b)}
\sum_{T\in \TD(\lm)} q^{|\lm|-\|T\|}.
\]
Note that $B_q(\lambda;a,b)$ is not necessarily a polynomial in $q$, but a
polynomial in $q^{1/2}$. In fact $B_q(\lambda;a,b)$ is a polynomial in $q$ if
and only if $a\equiv b\mod 2$.

We now state a generalization of Theorem~\ref{thm:1r}, or equivalently,
Theorem~\ref{thm:1}.
\begin{thm}\label{thm:gen}
  For $\lambda\in\Dyck(2n)$, nonnegative integers $a$ and $b$, we have
\[
B_q(\lambda;a,b)  
= \frac{\qint{n}!}{\prod_{c\in \C(\lambda)} \qint{|c|}} B_q(\Delta_n;a,b).
\]
\end{thm}

Note that if $a=b=0$ in Theorem~\ref{thm:gen}, we obtain Theorem~\ref{thm:1r}.
Although it is not necessary for our purpose, it is possible to find a formula
for $B_q(\Delta_n; a,b)$, see \eqref{eq:14}.

We will prove Theorem~\ref{thm:gen} in the next section. For the rest of this
section we prove several lemmas which are needed in the next section.

\begin{lem}\label{lem:layer}
 For $T\in\TD(\lm)$, every tile in $T$ lies between
$\lambda+(-i+1,i-1)$ and $\lambda+(-i,i)$ for some $i\geq0$. 
\end{lem}
\begin{proof}
  This lemma easily follows from the definition of truncated Dyck tilings.
\end{proof}

\begin{lem}\label{lem:split}
  Given Dyck paths $\lambda_1$, $\lambda_2$, $\lambda=\lambda_1*\lambda_2$, and
  lattice paths $\mu_1\in L(\lambda_1;a,i)$, $\mu_2\in L(\lambda_2;i,b)$, and
  $\mu=\mu_1*\mu_2$, there is a bijection
\[
\phi:\TD(\lm)\to \TD(\lambda_1/\mu_1)\times\TD(\lambda_2/\mu_2)
\]
such that if $\phi(T)=(T_1,T_2)$, then $\|T\|=\|T_1\|+\|T_2\|$.
\end{lem}
\begin{proof}
  We can find such a bijection $\phi$ naturally as follows. Suppose
  $\lambda_1\in\Dyck(2n_1)$ and $\lambda_2\in\Dyck(2n_2)$.  Let $O=(0,0),
  N=(n_1+n_2,n_1+n_2), A=O+(-a,a), B=N+(-b,b), P=(n_1,n_1), Q=P+(-i,i)$. For
  $T\in\TD(\lm)$, define $\phi(T)=(T_1,T_2)$ where $T_1$ and $T_2$ are the
  tilings of $\lambda_1/\mu_1$ and $\lambda_2/\mu_2$ obtained from $T$ by
  cutting the tiles of $T$ with the segment $PQ$, see Figure~\ref{fig:phi}. We
  need to show that $T_1$ and $T_2$ are truncated Dyck tilings of
  $\lambda_1/\mu_1$ and $\lambda_2/\mu_2$. Since the two conditions in
  Definition~\ref{defn:TD} are obvious, it is enough to show that each tile is a
  truncated Dyck tile.

  Suppose $\eta\in T$. If $\eta$ is not divided by the segment $PQ$, it is a
  truncated Dyck tile in $T_1$ or $T_2$. Otherwise, $\eta$ is divided into two
  tiles $\eta_1\in T_1$ and $\eta_2\in T_2$. Let $s_1$ and $s_2$ be the
  southwest cell and the northeast cell of $\eta$ respectively, and $s$ the cell
  where $\eta$ is divided by the segment $PQ$, see Figure~\ref{fig:eta}. In
  order to prove that $\eta_1$ and $\eta_2$ are truncated Dyck tiles, it
  suffices to show that $\HT(s)=\HT(s_1)$, where $\HT(s)$ is the distance
  between $s$ and the line $y=x$. Since $\eta$ is a truncated Dyck tile, we have
  $\HT(s)\geq\HT(s_1)$. On the other hand, by Lemma~\ref{lem:layer}, $\eta$ lies
  between $\lambda+(-i+1,i-1)$ and $\lambda+(-i,i)$ for some $i\geq0$. Since
  $\lambda$ touches the line $y=x$ at $P$, the cell $s$ has the minimal height
  among all cells between $\lambda+(-i+1,i-1)$ and $\lambda+(-i,i)$. Thus
  $\HT(s)\leq \HT(s_1)$, and we get $\HT(s)= \HT(s_1)$. This proves that
  $(T_1,T_2)\in \TD(\lambda_1/\mu_1)\times\TD(\lambda_2/\mu_2)$. Conversely, for
  such a pair $(T_1,T_2)$ we can construct $T$ by taking the union of $T_1$ and
  $T_2$ and attaching each two tiles if they share a border on the segment
  $PQ$. Thus $\phi$ is a bijection.  If $\phi(T)=(T_1,T_2)$, we clearly have
  $\|T\|=\|T_1\|+\|T_2\|$.
\end{proof}

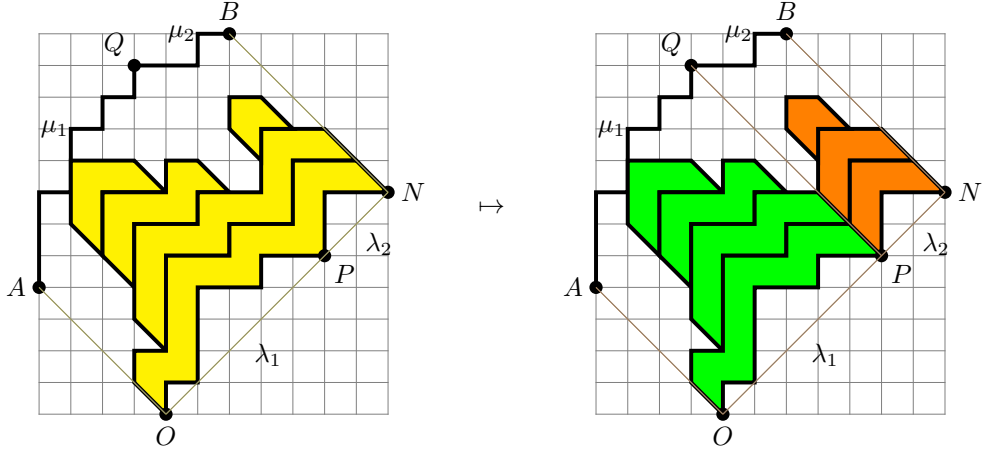
\begin{figure}
  \centering
  \begin{pspicture}(-1,-1)(12,13)
\psgrid(0,0)(11,12)
\psline(0,4)(0,7)(1,7)(1,9)(2,9)(2,10)(3,10)(3,11)(5,11)(5,12)(6,12)
\psset{fillstyle=solid, fillcolor=yellow}
\pspolygon(1,8)(3,8)(4,7)(2,7)(2,5)(1,6)(1,8)
\pspolygon(3,4)(2,5)(2,7)(4,7)(4,8)(5,8)(6,7)(5,7)(5,6)(3,6)(3,4)
\pspolygon(4,2)(3,3)(3,6)(5,6)(5,7)(7,7)(7,9)(9,9)(10,8)(8,8)(8,6)(6,6)(6,5)(4,5)(4,2)
\pspolygon(10,8)(8,8)(8,6)(6,6)(6,5)(4,5)(4,2)(3,2)(3,1)(4,0)(4,1)(5,1)(5,4)(7,4)(7,5)(9,5)(9,7)(11,7)(10,8)
\pspolygon(6,10)(7,10)(8,9)(7,9)(7,8)(6,9)(6,10)
\psdot(4,0)
\psdot(11,7)
\psdot(6,12)
\psdot(9,5)
\psdot(3,11)
\psdot(0,4)
\uput[270](4,0){$O$}
\uput[-45](9,5){$P$}
\uput[0](11,7){$N$}
\uput[90](6,12){$B$}
\uput[135](3,11){$Q$}
\uput[180](0,4){$A$}
\uput[-45](6.5,2.5){$\lambda_1$}
\uput[-45](10,6){$\lambda_2$}
\rput(.5, 9){$\mu_1$}
\rput(4.5, 12){$\mu_2$}
\psset{linewidth=.1pt,linecolor=gray}
\psline(4,0)(11,7)
\psline(4,0)(0,4)
\psline(11,7)(6,12)
\end{pspicture}
  \begin{pspicture}(0,-1)(4,13)
\rput(2,6.5){$\mapsto$}
\end{pspicture}
  \begin{pspicture}(-1,-1)(12,13)
\psgrid(0,0)(11,12)
\psline(0,4)(0,7)(1,7)(1,9)(2,9)(2,10)(3,10)(3,11)(5,11)(5,12)(6,12)
\psset{fillstyle=solid, fillcolor=green}
\pspolygon(1,8)(3,8)(4,7)(2,7)(2,5)(1,6)(1,8)
\pspolygon(3,4)(2,5)(2,7)(4,7)(4,8)(5,8)(6,7)(5,7)(5,6)(3,6)(3,4)
\pspolygon(4,2)(3,3)(3,6)(5,6)(5,7)(7,7)(8,6)(6,6)(6,5)(4,5)(4,2)
\pspolygon (8,6)(6,6)(6,5)(4,5)(4,2)(3,2)(3,1)(4,0)(4,1)(5,1)(5,4)(7,4)(7,5)(9,5)(8,6)
\psset{fillstyle=solid, fillcolor=orange}
\pspolygon(7,7)(7,9)(9,9)(10,8)(8,8)(8,6)(7,7)
\pspolygon(10,8)(8,8)(8,6)(9,5)(9,7)(11,7)(10,8)
\pspolygon(6,10)(7,10)(8,9)(7,9)(7,8)(6,9)(6,10)
\psdot(4,0)
\psdot(11,7)
\psdot(6,12)
\psdot(9,5)
\psdot(3,11)
\psdot(0,4)
\uput[270](4,0){$O$}
\uput[-45](9,5){$P$}
\uput[0](11,7){$N$}
\uput[90](6,12){$B$}
\uput[135](3,11){$Q$}
\uput[180](0,4){$A$}
\uput[-45](6.5,2.5){$\lambda_1$}
\uput[-45](10,6){$\lambda_2$}
\rput(.5, 9){$\mu_1$}
\rput(4.5, 12){$\mu_2$}

\psset{linewidth=.1pt,linecolor=gray}
\psline(4,0)(11,7)
\psline(4,0)(0,4)
\psline(11,7)(6,12)
\psline(9,5)(3,11)
\end{pspicture}
\caption{The definition of the map $\phi$.}
  \label{fig:phi}
\end{figure}

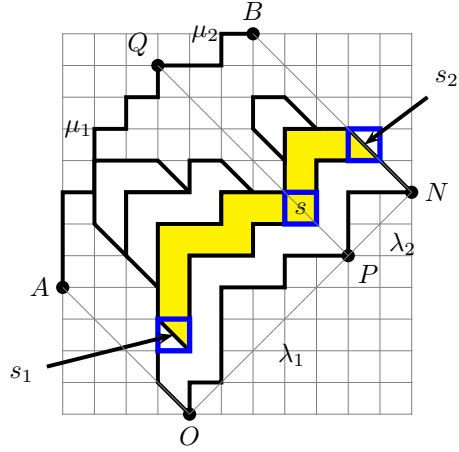
\begin{figure}
  \centering
  \begin{pspicture}(-1,-1)(12,13)
\psgrid(0,0)(11,12)
\psline(0,4)(0,7)(1,7)(1,9)(2,9)(2,10)(3,10)(3,11)(5,11)(5,12)(6,12)
\psset{fillstyle=solid}
\pspolygon(1,8)(3,8)(4,7)(2,7)(2,5)(1,6)(1,8)
\pspolygon(3,4)(2,5)(2,7)(4,7)(4,8)(5,8)(6,7)(5,7)(5,6)(3,6)(3,4)
\pspolygon[fillcolor=yellow](4,2)(3,3)(3,6)(5,6)(5,7)(7,7)(7,9)(9,9)(10,8)(8,8)(8,6)(6,6)(6,5)(4,5)(4,2)
\pspolygon(10,8)(8,8)(8,6)(6,6)(6,5)(4,5)(4,2)(3,2)(3,1)(4,0)(4,1)(5,1)(5,4)(7,4)(7,5)(9,5)(9,7)(11,7)(10,8)
\pspolygon(6,10)(7,10)(8,9)(7,9)(7,8)(6,9)(6,10)

\psset{fillstyle=none}
\psframe[linewidth=2pt, linecolor=blue,dimen=middle](3,2)(4,3)
\psframe[linewidth=2pt, linecolor=blue,dimen=middle](7,6)(8,7)
\psframe[linewidth=2pt, linecolor=blue,dimen=middle](9,8)(10,9)
\rput(7.5,6.5){$s$}
\psline{->}(-.5,1.5)(3.5,2.5)
\uput[200](-.5,1.5){$s_1$}
\psline{->}(11.5,10)(9.5,8.5)
\uput[50](11.5,10){$s_2$}

\psdot(4,0)
\psdot(11,7)
\psdot(6,12)
\psdot(9,5)
\psdot(3,11)
\psdot(0,4)
\uput[270](4,0){$O$}
\uput[-45](9,5){$P$}
\uput[0](11,7){$N$}
\uput[90](6,12){$B$}
\uput[135](3,11){$Q$}
\uput[180](0,4){$A$}
\uput[-45](6.5,2.5){$\lambda_1$}
\uput[-45](10,6){$\lambda_2$}
\rput(.5, 9){$\mu_1$}
\rput(4.5, 12){$\mu_2$}
\psset{linewidth=.1pt,linecolor=gray}
\psline(4,0)(11,7)
\psline(4,0)(0,4)
\psline(11,7)(6,12)
\psline(9,5)(3,11)
\end{pspicture}
 \caption{The truncated Dyck tile $\eta$ and the cells $s,s_1,s_2$.}
  \label{fig:eta}
\end{figure}%

Using Lemma~\ref{lem:split} one can easily obtain the following lemma.

\begin{lem}\label{lem:sep}
We have
\[
B_q(\lambda_1*\lambda_2; a,b) = \sum_{i\ge0} B_q(\lambda_1; a,i) B_q(\lambda_2; i,b).
\]
\end{lem}

% \begin{proof}
%   Suppose $\lambda_1\in\Dyck(2n_1)$ and $\lambda_2\in\Dyck(2n_2)$, and let
%   $\lambda=\lambda_1*\lambda_2$.  For $\mu\in L(\lambda;a,b)$, there is a unique
%   $i$ such that $\mu$ passes through $(n_1,n_1)+(-i,i)$. Let $L(\lambda;a,i,b)$ be
%   the set of $\mu\in L(\lambda;a,b)$ passing through $(n_1,n_1)+(-i,i)$.  Then
% \begin{equation}
%   \label{eq:12}
% B_q(\lambda;a,b) = \sum_{i\ge0} \sum_{\mu\in L(\lambda;a,i,b)}
% \sum_{T\in \TD(\lm)} q^{|\lm|-\|T\|}.
% \end{equation}

% Now we fix $\mu\in L(\lambda;a,i,b)$.  Let $O=(0,0), N=(n_1+n_2,n_1+n_2),
% A=O+(-a,a), B=N+(-b,b), P=(n_1,n_1), Q=P+(-i,i)$. Let $\mu_1$ (resp.~$\mu_2$) be
% the sub-path of $\mu$ from $A$ to $Q$ (resp.~from $Q$ to $B$), see
% Figure~\ref{fig:phi}.  Then $\mu_1\in L(\lambda_1;a,i)$, $\mu_2\in
% L(\lambda_2;i,b)$, and $\mu$ is determined by $\mu_1$ and $\mu_2$.

% Since $|\lm|=|\lambda_1/\mu_1|+|\lambda_2/\mu_2|$, by Lemma~\ref{lem:split}, we
% have
% \begin{equation}
%   \label{eq:11}
% \sum_{T\in \TD(\lm)} q^{|\lm|-\|T\|}  
% =\sum_{T_1\in \TD(\lambda_1/\mu_1)} q^{|\lambda_1/\mu_1|-\|T_1\|}
% \sum_{T_2\in \TD(\lambda_2/\mu_2)} q^{|\lambda_2/\mu_2|-\|T_2\|}.
% \end{equation}
% By \eqref{eq:12} and \eqref{eq:11}, we get
% \begin{align*}
% B_q(\lambda;a,b) &= \sum_{i\ge0} 
% \sum_{\substack{\mu_1\in L(\lambda_1;a,i)\\ \mu_2\in L(\lambda_2;i,b)}}
% \sum_{T_1\in \TD(\lambda_1/\mu_1)} q^{|\lambda_1/\mu_1|-\|T_1\|}
% \sum_{T_2\in \TD(\lambda_2/\mu_2)} q^{|\lambda_2/\mu_2|-\|T_2\|}\\
% &= \sum_{i\ge0} B_q(\lambda_1; a,i) B_q(\lambda_2; i,b),
% \end{align*}
% which finishes the proof.
% \end{proof}

We use the standard notations for $q$-binomial coefficients:
\[
\qbinom{n}{k} = \frac{\qint{n}!}{\qint{k}!\qint{n-k}!},
\qquad \qbinom{n_1+\cdots+n_k}{n_1,\dots,n_k} 
= \frac{\qint{n_1+\cdots+n_k}!}{\qint{n_1}!\cdots\qint{n_k}!}.
\]

\begin{lem}\label{lem:tile}
  Let $\mu$ be a lattice path in $L(\Delta_n; a,b)$ passing through $P+(-t,t)$
  for some integer $t\ge0$, where $P=(0,n)$, the peak of $\Delta_n$.  Then
\[
\sum_{T\in\TD(\Delta_{n}/\mu)}
q^{\|T\|} =\qbinom{n+t}{n}.
\]  
\end{lem}
\begin{proof}
  Let $T\in\TD(\Delta_{n}/\mu)$. By Lemma~\ref{lem:layer}, every tile $\eta$ in
  $T$ lies between $\Delta_{n}+(-i+1,i-1)$ and $\Delta_{n}+(-i,i)$ for some
  $i\in[t]$. Moreover, $\eta$ is the unique tile between
  $\Delta_{n}+(-i+1,i-1)$ and $\Delta_{n}+(-i,i)$ because $\Delta_n$ has only
  one peak.

  For $i\in[t]$, let $h_i$ be the half-length of the tile in $T$ between
  $\Delta_{n}+(-i+1,i-1)$ and $\Delta_{n}+(-i,i)$. If there is no such tile, we
  define $h_i=0$. Then $\nu=(h_1,\dots,h_t)$ is a partition contained in a
  $t\times n$ box, and $\|T\|=|\nu|$, see Figure~\ref{fig:tpartition}. This gives
  a bijection between $\TD(\Delta_{n}/\mu)$ and the set of partitions contained
  in a $t\times n$ box. It is well-known that the sum of $q^{|\nu|}$ for such
  partitions $\nu$ is equal to the right hand side, see
  \cite[1.7.3~Proposition]{EC1}.

 \begin{figure}
    \centering
\psset{unit=20pt}
\begin{pspicture}(0,1)(6,7) 
\psgrid[subgriddiv=2](0,1)(6,7)
\psline(3,1)(3,4)(6,4)
\psline[linewidth=.1pt, linecolor=gray](3,1)(6,4)
\psline(6,4)(3,7)
\psline(0,4)(3,1)
\rput(1,-1){\psline(-1,5)(-1,6.5)(-.5,6.5)(-.5,7)(0,7)
\psline(0,7)(0,7.5)(1,7.5)(1,8)(2,8)}
\uput[135](1,6){$P+(-t,t)$}
\uput[-45](3,4){$P$}
\rput(4,3){$\Delta_{n}$}
{
\psset{fillstyle=solid,fillcolor=yellow,unit=10pt}
\rput(6,8){\pspolygon(0,0)(0,-5)(-1,-4)(-1,1)(4,1)(5,0)}
\rput(5,9){\pspolygon(0,0)(0,-4)(-1,-3)(-1,1)(3,1)(4,0)}
\rput(4,10){\pspolygon(0,0)(0,-2)(-1,-1)(-1,1)(1,1)(2,0)}
\psset{linecolor=blue}
\rput(6,8){\psline(0,0)(6,0)}
\rput(5,9){\psline(0,0)(6,0)}
\rput(4,10){\psline(0,0)(6,0)}
\rput(3,11){\psline(0,0)(6,0)}
\rput(2,12){\psline(0,0)(6,0)}
\rput(6,8){\psline(0,0)(-4,4)}
\rput(7,8){\psline(0,0)(-4,4)}
\rput(8,8){\psline(0,0)(-4,4)}
\rput(9,8){\psline(0,0)(-4,4)}
\rput(10,8){\psline(0,0)(-4,4)}
\rput(11,8){\psline(0,0)(-4,4)}
\rput(12,8){\psline(0,0)(-4,4)}
}
\psdot(3,4) \psdot(1,6)
\end{pspicture}
\caption{A truncated Dyck tiling of $\Delta_n/\mu$ corresponding to the
  partition $(5,4,2,0)$. Here an additional grid is drawn to visualize the
  partition.}
    \label{fig:tpartition}
  \end{figure}
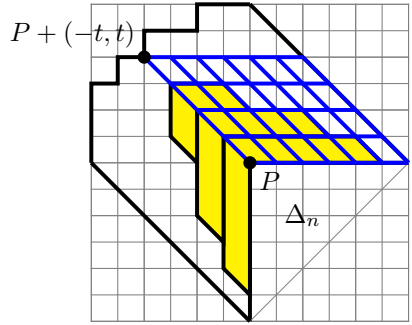

\end{proof}

\section{Proof of Theorem~\ref{thm:gen}}
\label{sec:proof1}

In this section we prove Theorem~\ref{thm:gen} in three steps. In the first
step we prove the theorem in the case $\lambda=\Delta_{n_1,\dots,n_k}$ and
$a=b=0$. In the second step we prove theorem in the case
$\lambda=\Delta_{n_1,\dots,n_k}$, and $a$ and $b$ are arbitrary. In the third
step we prove the theorem without restrictions.

\subsection{Step 1: $\lambda=\Delta_{n_1,\dots,n_k}$ and $a=b=0$.}

In this subsection we prove Theorem~\ref{thm:gen} for
$\lambda=\Delta_{n_1,\dots,n_k}\in\Dyck(2n)$ and $a=b=0$. In other words, we
show that
\begin{equation}
  \label{eq:16}
B_q(\Delta_{n_1,\dots,n_k}; 0,0) = \qbinom{n}{n_1,\dots,n_k}.
\end{equation}

Throughout this subsection $\lambda$ denotes $\Delta_{n_1,\dots,n_k}$ and for
$i\in[k]$, $P_i$ denotes the peak of the $i$th sub-Dyck path $\Delta_{n_i}$,
i.e.
\[
P_i=(n_1+\cdots+n_{i-1}, n_1+\cdots+n_{i}).
\]

Consider a lattice path $\mu\in L(\lambda;0,0)$.  For each $i\in[k]$, we can
find the intersection $Q_i$ of $\mu$ and the line with slope $-1$ passing
through $P_i$. Then we have $Q_i=P_i + (-t_i,t_i)$ for some integer
$t_i\geq0$. Note that $t_1=t_{k}=0$.  We define $L'(\lambda;t_1,\dots,t_k)$ to
be the set of such lattice paths $\mu$.  Then,
\begin{equation}
  \label{eq:3}
B_q(\lambda;0,0) = 
\sum_{\substack{t_1,\dots,t_k\ge0 \\ t_1=t_k=0}}
\sum_{\mu\in L'(\lambda;t_1,\dots,t_k)}
q^{|\lm|}
\sum_{T\in\TD(\lm)}
q^{-\|T\|}.
\end{equation}

\begin{figure}
  \centering
\psset{unit=15pt}
\begin{pspicture}(0,0)(9,10)
\psline(0,0)(0,2)(2,2)(2,5)(5,5)(5,7)(7,7)(7,9)(9,9)

\psline[linecolor=red](0,0)(0,3)(.5,3)(.5,3.5)
\psline[linecolor=blue](.5,3.5)(.5,4)(1,4)(1,6)(3,6)(3,7)
\psline[linecolor=orange](3,7)(3,8)(5.5,8)(5.5,8.5)
\psline[linecolor=gray](5.5,8.5)(6.5,8.5)(6.5,9)(9,9)

\rput(-.5,3){$\mu_1$}
\rput(1.5,6.5){$\mu_2$}
\rput(3,8.5){{$\mu_3$}}
\rput(6,9.5){{$\mu_4$}}

\psdot(0,2) \psdot(2,5) \psdot(1,6) \psdot(5,7) \psdot(4,8) \psdot(7,9)
\psset{linewidth=1pt,linecolor=gray}
\psline(0,0)(9,9)
\psline(2,2)(.5,3.5)
\psline(5,5)(3,7)
\psline(7,7)(5.5,8.5)
\psline(2,5)(1,6)
\psline(5,7)(4,8)
\rput[l](.5,1.5){$P_1=Q_1$}
\rput(2.5,4.5){$P_2$} \rput(1.5,5.5){$Q_2$}
\rput(5.5,6.5){$P_3$} \rput(4.5,7.5){$Q_3$}
\rput[l](7.5,8.5){$P_4=Q_4$}
\rput(1.5,3){$\ell_1$}
\rput[bl](4,6){$\ell_2$}
\rput(6.5,8){$\ell_3$}
\end{pspicture}
 \caption{Dividing $\mu$ into $k$ sub-paths for $k=4$.}
 \label{fig:divmu}
\end{figure}
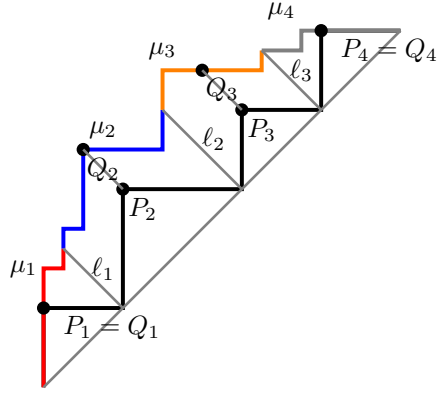

Suppose $\mu\in L'(\lambda;t_1,\dots,t_k)$. For $i\in[k-1]$, let $\ell_i$ be the
line with slope $-1$ passing through $(n_1+\cdots+n_{i} ,n_1+\cdots+n_{i})$, the
ending point of $\Delta_{n_i}$. Let $\mu_1,\dots,\mu_k$ be the paths obtained by
dividing $\mu$ using the lines $\ell_1,\dots,\ell_{k-1}$, see
Figure~\ref{fig:divmu}.  By Lemma~\ref{lem:split}, we have
\[
\sum_{T\in\TD(\lm)} q^{-\|T\|}
= \prod_{i=1}^k\sum_{T\in\TD(\Delta_{n_i}/\mu_i)} q^{-\|T\|}.
\]
Since $\mu_i$ passes through $Q_i=P_i+(-t_i,t_i)$, by Lemma~\ref{lem:tile}, we
have
\[
\sum_{T\in\TD(\Delta_{n_i}/\mu_i)} q^{-\|T\|} = \Qbinom{n_i+t_i}{n_i}{q^{-1}}.
\]
Thus \eqref{eq:3} can be written as
\begin{equation}
  \label{eq:4}
B_q(\lambda;0,0) = 
\sum_{\substack{t_1,\dots,t_k\\ t_1=t_k=0}}
\prod_{i=1}^k \Qbinom{n_i+t_i}{n_i}{q^{-1}}
\sum_{\mu\in L'(\lambda;t_1,\dots,t_k)}
q^{|\lm|}.
\end{equation}

The latter sum in \eqref{eq:4} can be computed as follows. 

\begin{lem}\label{lem:area}
Suppose $t_1=t_k=0$. Then
\[
\sum_{\mu\in L'(\lambda;t_1,\dots,t_k)} q^{|\lm|}
= \prod_{i=1}^{k-1} q^{n_i t_i + n_{i+1}t_{i+1}+ \frac12 (t_i-t_{i+1})^2} 
\qbinom{n_i+n_{i+1}}{n_i+t_i-t_{i+1}}.
\]
\end{lem}
\begin{proof}
  Let $\mu\in L'(\lambda;t_1,\dots,t_k)$.  Then $\mu$ passes through the points
  $Q_i=P_i+(-t_i,t_i)$ for $i=1,2,\dots,k$.  For $i=1,2,\dots,k-1$, we define
  $\nu_i$ to be the sub-path of $\mu$ from $Q_i$ to $Q_{i+1}$, and $R_i$ to be
  the region bounded by $\nu_i$, $P_iQ_i$, $P_{i+1}Q_{i+1}$ and $\lambda$. Since
  $t_1=t_k=0$, $|\lm|$ is the sum of the areas of $R_1,\dots,R_{k-1}$. We can
  divide the region $R_i$ as shown in Figure~\ref{fig:division}.  In such a
  division, the area of region \textbf{1} (resp.~region \textbf{2}) is $n_it_i$
  (resp.~$n_{i+1}t_{i+1}$). Since region \textbf{3} is an isosceles right
  triangle such that the length of the hypotenuse is $\sqrt2|t_i-t_{i+1}|$, the
  area of region \textbf{3} is equal to $\frac12 (t_1-t_{i+1})^2$.  If we add
  $q$ raised to the area of region \textbf{4} for all possible lattice paths
  $\nu_i$ from $Q_i$ to $Q_{i+1}$, we get
  $\qbinom{n_{i}+n_{i+1}}{n_i+t_i-t_{i+1}}$.  Summing these results we
  obtain the lemma.

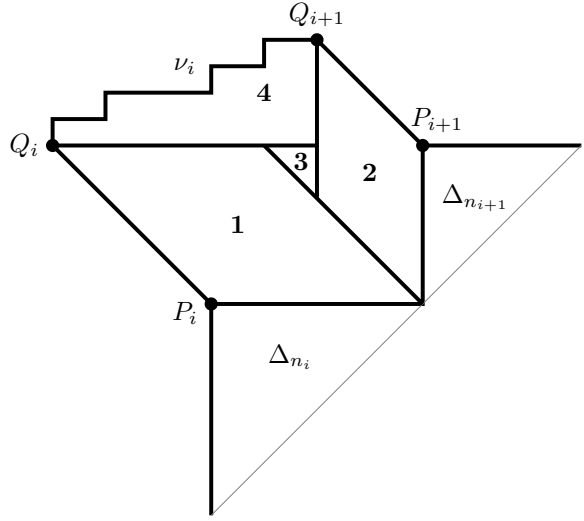
\begin{figure}
  \centering
\psset{unit=20pt}  
\begin{pspicture}(0,0)(10,10)
\psline(3,0)(3,4)(7,4)(7,7)(10,7)
\psline[linewidth=.1pt, linecolor=gray](3,0)(10,7)
\psline(3,4)(0,7) \psline(7,4)(4,7) \psline(7,7)(5,9)
\psline(5,9)(5,6) \psline(0,7)(5,7)
\psline(0,7)(0,7.5)(1,7.5)(1,8)(3,8)(3,8.5)(4,8.5)(4,9)(5,9)
\psdot(0,7) \psdot(5,9) \psdot(3,4) \psdot(7,7)
\uput[180](0,7){$Q_i$}
\uput[90](5,9){$Q_{i+1}$}
\uput[200](3,4){$P_{i}$}
\uput[70](7,7){$P_{i+1}$}
\rput(2.5,8.5){$\nu_i$}
\rput(4.5,3){$\Delta_{n_i}$}
\rput(8,6){$\Delta_{n_{i+1}}$}
\rput(3.5,5.5){\bf 1}
\rput(6,6.5){\bf 2}
\rput(4.7,6.7){\bf 3}
\rput(4,8){\bf 4}
\end{pspicture}
\caption{Dividing the region $R_i$ into four regions.}
  \label{fig:division}
\end{figure}

\end{proof}

Since $t_1=t_k=0$ and $q^{n_it_i}\Qbinom{n_i+t_i}{n_i}{q^{-1}} =
\qbinom{n_i+t_i}{n_i}$, by \eqref{eq:4} and Lemma~\ref{lem:area}, we have
\begin{align*}
B_q(\lambda;0,0) &= 
\sum_{\substack{t_1,\dots,t_k\ge0 \\ t_1=t_k=0}}
\prod_{i=1}^{k-1} q^{n_{i+1}t_{i+1}+ \frac12 (t_i-t_{i+1})^2} 
\qbinom{n_i+t_i}{n_i}
\qbinom{n_i+n_{i+1}}{n_i+t_i-t_{i+1}}\\
&= \sum_{\substack{t_1,\dots,t_k\ge0 \\ t_1=t_k=0}}
\prod_{i=1}^{k-1} q^{t_{i+1}(n_{i+1}+ t_{i+1} - t_i)}
\qbinom{n_i+t_i}{n_i}
\qbinom{n_i+n_{i+1}}{n_i+t_i-t_{i+1}}, 
\end{align*}
where the following equality is used:
\[
\sum_{i=1}^{k-1} \frac12 (t_i-t_{i+1})^2  =
\sum_{i=1}^{k-1} \left(t_{i+1}^2 - t_it_{i+1}\right).
\]

Now \eqref{eq:16} follows from the lemma below.

\begin{lem}\label{lem:formula}
For integers $k\geq1$, and $n_1,\dots,n_{k}\geq0$, we have
\begin{equation}
  \label{eq:18}
\sum_{\substack{t_1,\dots,t_k\ge0 \\ t_1=t_k=0}} \prod_{i=1}^{k-1}
q^{t_{i+1}(n_{i+1} - t_i + t_{i+1})}
\qbinom{n_i+t_i}{n_i} \qbinom{n_i+n_{i+1}}{n_i+t_i-t_{i+1}}
= \qbinom{n_1+\cdots+n_{k}}{n_1,\dots,n_{k}}.
\end{equation}
\end{lem}
\begin{proof}
  This can be done in a straightforward manner by induction on $k$ using the
  $q$-Chu-Vandermonde identity (see \cite[page 190, Solution to Exercise~100 in
  Chapter~1]{EC1}):
\[
\sum_{i\ge0} q^{i(m-k+i)}\qbinom{m}{k-i} \qbinom{n}{i}=\qbinom{m+n}{k}.
\]
\end{proof}

\subsection{Step 2: $\lambda=\Delta_{n_1,\dots,n_k}$ and $a,b$ are arbitrary.}

In this subsection we prove Theorem~\ref{thm:gen} when
$\lambda=\Delta_{n_1,\dots,n_k}\in\Dyck(2n)$, and $a$ and $b$ are arbitrary.  In
other words, we show that
\begin{equation}
  \label{eq:1}
B_q(\Delta_{n_1,\dots,n_k}; a,b) = \qbinom{n}{n_1,\dots,n_k} B_q(\Delta_n;a,b).
\end{equation}

We will prove \eqref{eq:1} by induction on $(a,b)$.  We have showed this when
$(a,b)=(0,0)$ in Step 1.  Let $(a,b)\ne(0,0)$ and suppose \eqref{eq:1} is true
for all pairs $(a',b')\ne(a,b)$ with $a'\leq a$ and $b'\leq b$. By symmetry we
can assume $a\ne0$.

Consider the two Dyck paths $\Delta_{a, n_1,\dots,n_k}$ and $\Delta_{a, n}$.
By the induction hypothesis, we have
\[
B_q(\Delta_{a, n_1,\dots,n_k}; 0,b) = \qbinom{a+n}{a,n_1,\dots,n_k} B_q(\Delta_{a+n};0,b), 
\]
\[
B_q(\Delta_{a, n}; 0,b) = \qbinom{a+n}{n} B_q(\Delta_{a+n};0,b).
\]
Combining the above two equations we get
\begin{equation}
  \label{eq:5}
B_q(\Delta_{a, n_1,\dots,n_k}; 0,b) =  \qbinom{n}{n_1,\dots,n_k } B_q(\Delta_{a,n};0,b).  
\end{equation}

\begin{lem}\label{lem:induct}
We have
\[
B_q(\Delta_{a}*\lambda; 0,b) = 
\sum_{i=0}^a q^{i^2/2} \qbinom{a}{i} B_q(\lambda; i,b). 
\]  
\end{lem}
\begin{proof}
By Lemma~\ref{lem:sep}, we have
\[
B_q(\Delta_{a}*\lambda; 0,b) = \sum_{i\ge0} B_q(\Delta_a; 0,i) B_q(\lambda; i,b). 
\]  
Since $B_q(\Delta_a; 0,i)=q^{i^2/2} \qbinom{a}{i}$, we are done.
\end{proof}

By Lemma~\ref{lem:induct} we have
\begin{align}
  \label{eq:7}
B_q(\Delta_{a, n_1,\dots,n_k}; 0,b) &=
\sum_{i=0}^a q^{i^2/2} \qbinom{a}{i} B_q(\Delta_{n_1,\dots,n_k}; i,b),\\
\label{eq:13}
B_q(\Delta_{a, n}; 0,b) &= 
\sum_{i=0}^a q^{i^2/2} \qbinom{a}{i} B_q(\Delta_{n}; i,b).
\end{align}
By \eqref{eq:5}, \eqref{eq:7}, and \eqref{eq:13} we get
\begin{equation}
  \label{eq:6}
\sum_{i=0}^a q^{i^2/2} \qbinom{a}{i} B_q(\Delta_{n_1,\dots,n_k}; i,b)
=  \qbinom{n}{n_1,\dots,n_k}\sum_{i=0}^a q^{i^2/2} \qbinom{a}{i} B_q(\Delta_{n}; i,b).  
\end{equation}
By the induction hypothesis, for all $i<a$, we have
\[
B_q(\Delta_{n_1,\dots,n_k}; i,b)
=\qbinom{n}{n_1,\dots,n_k} B_q(\Delta_{n}; i,b).
\]
Thus the summands in both sides of \eqref{eq:6} equal for all $i<a$, forcing the
summands for $i=a$ to be equal as well. This implies that
\[
q^{a^2/2} B_q(\Delta_{n_1,\dots,n_k}; a,b)
=\qbinom{n}{n_1,\dots,n_k} q^{a^2/2} B_q(\Delta_{n}; i,b).
\]
Thus we have that \eqref{eq:1} is also true for $(a,b)$, and by induction, we
are done.

In particular, if $k=2$, we have the following.
\begin{prop}\label{prop:two}
We have
\[
B_q(\Delta_{n_1}*\Delta_{n_2}; a,b) = \qbinom{n_1+n_2}{n_1}B_q(\Delta_{n_1+n_2};
a,b). 
\]  
\end{prop}

\subsection{Step 3: Without restrictions}

In this subsection we prove Theorem~\ref{thm:gen} without restrictions. To do
this we need another lemma.  For a lattice path $\nu$, we define $\nu^-$ to be
the lattice path obtained from $\nu$ by deleting the first step and the last
step.

\begin{lem}\label{lem:lambda-}
  If $\lambda\in\Dyck(2n)$ cannot be expressed as $\lambda_1*\lambda_2$, we
  have
\[
B_q(\lambda; a,b) = \sum_{i\ge0} \sum_{0\leq r,s\leq1} 
 q^{(n-2)i+a+b-(r+s)/2} B_q(\lambda^-; a-i-r, b-i-s).
\]
\end{lem}
\begin{proof}
  Let $\mu\in L(\lambda;a,b)$, $T\in \TD(\lm)$, and $O=(0,0), N=(n,n),
  A=O+(-a,a), B=N+(-b,b)$.

  \begin{figure}
    \centering
    \begin{pspicture}(0,-1)(9,9)
      \psgrid(0,0)(9,8)
      \psline(0,5)(1,5)(1,7)(2,7)(2,8)(5,8)
      \psline(9,4)(7,4)(7,3)(5,3)(5,0)
\psset{fillstyle=solid,fillcolor=yellow}
\pspolygon(9,4)(7,4)(7,3)(5,3)(5,0)(4,1)(4,4)(6,4)(6,5)(8,5)
\rput(-1,1){\pspolygon(9,4)(7,4)(7,3)(5,3)(5,0)(4,1)(4,4)(6,4)(6,5)(8,5)}
\pspolygon(2,4)(2,6)(4,6)(4,7)(5,7)(6,6)(5,6)(5,5)(3,5)(3,3)
\pspolygon(1,6)(1,7)(2,7)(3,6)(2,6)(2,5)
\rput(2,1){\pspolygon(1,6)(1,7)(2,7)(3,6)(2,6)(2,5)}
\pspolygon[linewidth=3pt,fillstyle=none,linecolor=red]
(5,6)(5,5)(3,5)(3,3)(1,5)(1,7)(2,7)(2,8)(4,8)(6,6)

\psdot(5,0) \psdot(9,4) \psdot(0,5) \psdot(5,8)
\psdot(3,3) \psdot(6,6) \psdot(1,5) \psdot(4,8)
\uput[-135](3,3){$O'$} \uput[45](6,6){$N'$} \uput[-100](1,5){$A'$} \uput[90](4,8){$B'$}

\psset{linewidth=.1pt,linecolor=gray}
\psline(5,0)(0,5) \psline(9,4)(5,8) \psline(5,0)(9,4)
\uput[-90](5,0){$O$}
\uput[0](9,4){$N$}
\uput[180](0,5){$A$}
\uput[90](5,8){$B$}
  \end{pspicture}
  \caption{The tiling $T'$ is a truncated Dyck tiling of the region
    $\lambda'/\mu^-$ whose boundary is drawn with thick lines.}
    \label{fig:lambda-}
  \end{figure}
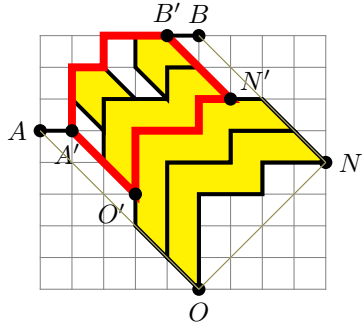

  Suppose $T$ has exactly $i$ tiles of length $2n$. Since $\lambda$ cannot be
  expressed as $\lambda_1*\lambda_2$, we have $\lambda^-\in\Dyck(2n-2)$.  Let
  $\lambda' = \lambda^- +(-i,i)$. We denote the starting point and the ending
  point of $\lambda'$ (resp.~$\mu^-$) by $O'$ and $N'$ (resp.~$A'$ and $B'$),
  see Figure~\ref{fig:lambda-}.  Then $A'=O'+(-a+i+r,a-i-r)$ and
  $B'=N'+(-b-i+s,b+i-s)$ for some $r,s\in \{0,1\}$ depending on $\mu$. Note that
  $\mu^-\in L(\lambda'; a-i-r, b-i-s)$.  Let $T'$ be the set of tiles in $T$
  except the $i$ tiles of length $2n$. Then we can consider $T'$ as a tiling in
  $\TD(\lambda'/\mu^-)$, or by translating it by $(i,-i)$, a tiling in
  $\TD(\lambda^-/\mu')$, where $\mu'=\mu^- +(i,-i)\in L(\lambda^-;a-i-r,b-i-s)$.
  Note that $T$ is determined by $i$ and $T'$.  It is easy to check that
\begin{align*}
  |\lm| &= |\lambda^-/\mu'| +2(n-1)i +a+b-(r+s)/2,\\
  \|T\| &= \|T'\| + ni.
\end{align*}
Thus,
\begin{align*}
B_q(\lambda; a,b) &= \sum_{\mu\in L(\lambda;a,b)} \sum_{T\in\TD(\lm)}
q^{|\lm|-\|T\|}\\
&=\sum_{i\ge0} \sum_{0\leq r,s\leq1} 
\sum_{\mu'\in L(\lambda^-;a-i-r,b-i-s)} 
\sum_{T'\in  \TD(\lambda^-/\mu')} q^{|\lambda^-/\mu'|-\|T'\|+(n-2)i+a+b-(r+s)/2}\\
&=\sum_{i\ge0} \sum_{0\leq r,s\leq1} 
 q^{(n-2)i+a+b-(r+s)/2} B_q(\lambda^-; a-i-r, b-i-s).
\end{align*}
\end{proof}

We now prove Theorem~\ref{thm:gen} by induction on $n$. If $n=0$, it is clear.
Suppose $n>0$ and the theorem is true for all integers less than $n$.

Case 1: $\lambda$ can be written as $\lambda_1*\lambda_2$. Suppose
$\lambda_1\in\Dyck(2n_1)$ and $\lambda_2\in\Dyck(2n_2)$. Then $n=n_1+n_2$. By
Lemma~\ref{lem:sep}, we have
\begin{equation}
  \label{eq:8}
 B_q(\lambda; a,b) = \sum_{i\ge0}  B_q(\lambda_1; a,i) B_q(\lambda_2; i,b).
\end{equation}
Since both $n_1$ and $n_2$ are smaller than $n$, by the induction hypothesis, we
have
\begin{align}\label{eq:9}
B_q(\lambda_1; a,i) &= 
\frac{\qint{n_1}!}{\prod_{c\in \C(\lambda_1)} \qint{|c|}} B_q(\Delta_{n_1}; a,i),\\
\label{eq:10}
B_q(\lambda_2; i,b) &= 
\frac{\qint{n_2}!}{\prod_{c\in \C(\lambda_2)} \qint{|c|}} B_q(\Delta_{n_2}; i,b).
\end{align}
By \eqref{eq:8}, \eqref{eq:9}, \eqref{eq:10}, and the fact that $\C(\lambda) =
\C(\lambda_1) \uplus \C(\lambda_2)$, we have
\begin{align*}
  B_q(\lambda; a,b) &= \frac{\qint{n_1}!\qint{n_2}!}{\prod_{c\in \C(\lambda)}
    \qint{|c|}}
  \sum_{i\ge0}  B_q(\Delta_{n_1}; a,i) B_q(\Delta_{n_2}; i,b)\\
  &=\frac{\qint{n_1}!\qint{n_2}!}{\prod_{c\in \C(\lambda)} \qint{|c|}}
  B_q(\Delta_{n_1,n_2}; a,b) & \mbox{(by  Lemma~\ref{lem:sep})}\\
  &=\frac{\qint{n_1}!\qint{n_2}!}{\prod_{c\in \C(\lambda)} \qint{|c|}}
  \qbinom{n_1+n_2}{n_1}B_q(\Delta_{n}; a,b) & \mbox{(by  Proposition~\ref{prop:two})}\\
  &=\frac{\qint{n}!}{\prod_{c\in \C(\lambda)} \qint{|c|}} B_q(\Delta_{n}; a,b).
\end{align*}

Case 2: $\lambda$ cannot be expressed as $\lambda_1*\lambda_2$.  Then
$\lambda^-\in\Dyck(2n-2)$ and
\[
\{|c|:c\in\C(\lambda)\}=\{|c|:c\in\C(\lambda^-)\}\cup\{n\}.
\]
Thus $B_q(\lambda; a,b)$ is equal to
\begin{align*}
& \sum_{i\ge0} \sum_{0\leq r,s\leq1} 
 q^{(n-2)i+a+b-(r+s)/2} B_q(\lambda^-; a-i-r, b-i-s) 
&\mbox{(by Lemma~\ref{lem:lambda-})}\\
&= \frac{\qint{n-1}!}{\prod_{c\in \C(\lambda^-)} \qint{|c|}}
\sum_{i\ge0} \sum_{0\leq r,s\leq1}  q^{(n-2)i+a+b-(r+s)/2}
B_q(\Delta_{n-1}; a-i-r, b-i-s) 
&\mbox{(by ind. hyp.)}\\
&= \frac{\qint{n}!}{\prod_{c\in \C(\lambda)} \qint{|c|}}
B_q(\Delta_{n}; a, b) .
&\mbox{(by Lemma~\ref{lem:lambda-})}
\end{align*}

Since Theorem~\ref{thm:gen} is true for $n$ in both cases, by induction we are done.

\section{Another proof of Proposition~\ref{prop:two}}
\label{sec:anoth-proof-prop}

The reader may notice that in the proof of Theorem~\ref{thm:gen} all we need in
Steps 1 and 2 is Proposition~\ref{prop:two}.  In this section we give another
proof of Proposition~\ref{prop:two} using hypergeometric series. 

Suppose $\mu\in L(\Delta_{n_1,n_2}; a,b)$.  Let $P_1$ and $P_2$ be the peaks of
$\Delta_{n_1}$ and $\Delta_{n_2}$. Then there are unique $i\geq0$ and $j\geq0$
such that $\mu$ passes through $Q_1=P_1+(-i,i)$ and $Q_2=P_2+(-j,j)$. Let
$\mu_1,\mu_2,\mu_3$ be the paths obtained from $\mu$ by dividing it at $Q_1$ and
$Q_2$. We can divide the region $\Delta_{n_1,n_2}/\mu$ as shown in
Figure~\ref{fig:division2}. Then
\[
\area(\mathbf{1})=\area(\mathbf{2})=n_1 i, \qquad
\area(\mathbf{3})=\area(\mathbf{4})=n_2j ,
\]
\[
\area(\mathbf{5})= \frac12(a-i)^2, \qquad
\area(\mathbf{6})= \frac12(i-j)^2, \qquad
\area(\mathbf{7})= \frac12(b-j)^2.
\]
Once $i$ and $j$ are fixed, the sums of $\area(\mathbf{8})$,
$\area(\mathbf{9})$, and $\area(\mathbf{10})$ for all possible $\mu_1$, $\mu_2$,
and $\mu_3$ are respectively $\qbinom{n_1}{a-i}$, $\qbinom{n_1+n_2}{n_1+i-j}$,
and $\qbinom{n_2}{b-j}$. Let $\nu_1$ and $\nu_2$ be the lattice paths obtained
by dividing $\mu$ with the line of slope $-1$ passing through $(n_1,n_1)$. Then
$\mu=\nu_1*\nu_2$. By Lemmas~\ref{lem:split} and \ref{lem:tile}, we have
\begin{align*}
\sum_{T\in\TD(\Delta_{n_1,n_2}/\mu)} q^{-\|T\|} &= 
\sum_{T\in\TD(\Delta_{n_1}/\nu_1)} q^{-\|T\|} \sum_{T\in\TD(\Delta_{n_2}/\nu_2)} q^{-\|T\|}\\
&=\Qbinom{n_1+i}{n_1}{q^{-1}} \Qbinom{n_2+j}{n_2}{q^{-1}}
=q^{-n_1i-n_2j}\qbinom{n_1+i}{n_1} \qbinom{n_2+j}{n_2}.
\end{align*}
Thus $B_q(\Delta_{n_1,n_2}; a,b)$ is equal to
\begin{align*}
&\sum_{i,j\ge0} q^{n_1i+n_2j+\frac12(a-i)^2+\frac12(i-j)^2+\frac12(b-j)^2}
\qbinom{n_1}{a-i}\qbinom{n_1+n_2}{n_1+i-j}\qbinom{n_2}{b-j}
\qbinom{n_1+i}{n_1} \qbinom{n_2+j}{n_2}\\
&= \sum_{i,j\ge0} q^{\frac{a^2+b^2}2 +(n_1-a)i+(n_2-b)j-ij+i^2+j^2}
\qbinom{n_1}{a-i}\qbinom{n_1+n_2}{n_1+i-j}\qbinom{n_2}{b-j}
\qbinom{n_1+i}{n_1} \qbinom{n_2+j}{n_2}. 
\end{align*}
Similarly, one can check that
\begin{equation}
  \label{eq:14}
B_q(\Delta_{n}; a,b)=
\sum_{i\ge0} q^{\frac{a^2+b^2}2 + (n-a-b)i + i^2}
\qbinom{n+i}i \qbinom{n}{a-i} \qbinom{n}{b-i}.
\end{equation}

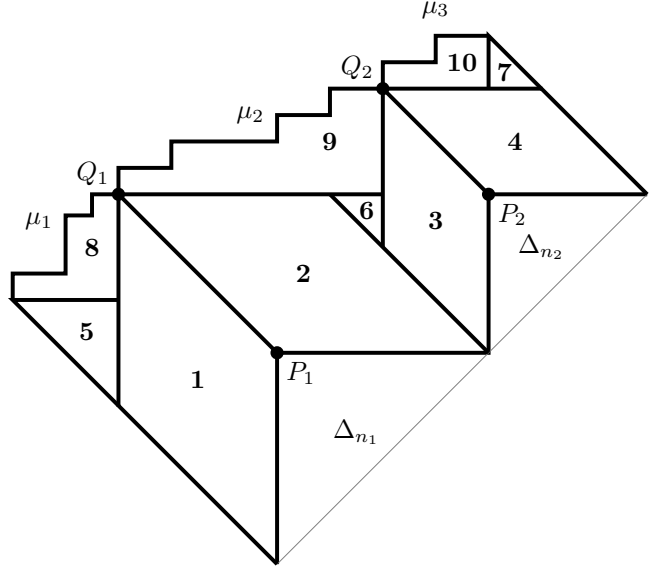
\begin{figure}
  \centering
\psset{unit=20pt}  
\begin{pspicture}(-2,0)(10,11)
\psline(3,0)(3,4)(7,4)(7,7)(10,7)
\psline[linewidth=.1pt, linecolor=gray](3,0)(10,7)
\psline(3,4)(0,7) \psline(7,4)(4,7) \psline(7,7)(5,9)
\psline(5,9)(5,6) \psline(0,7)(5,7)
\psline(0,7)(0,7.5)(1,7.5)(1,8)(3,8)(3,8.5)(4,8.5)(4,9)(5,9)
\psdot(0,7) \psdot(5,9) \psdot(3,4) \psdot(7,7)
\psline(3,0)(0,3)(0,7)
\psline(5,9)(8,9)(10,7)
\psline(0,3)(-2,5)(0,5)
\psline(8,9)(7,10)(7,9)
\psline(-2,5)(-2,5.5)(-1,5.5)(-1,6.6)(-.5,6.6)(-.5,7)(0,7)
\psline(5,9)(5,9.5)(6,9.5)(6,10)(7,10)
\uput[135](0,7){$Q_1$}
\uput[135](5,9){$Q_{2}$}
\uput[-45](3,4){$P_{1}$}
\uput[-45](7,7){$P_{2}$}
\rput(-1.5,6.5){$\mu_1$}
\rput(2.5,8.5){$\mu_2$}
\rput(6,10.5){$\mu_3$}
\rput(4.5,2.5){$\Delta_{n_1}$}
\rput(8,6){$\Delta_{n_{2}}$}
\rput(1.5,3.5){\bf 1}
\rput(3.5,5.5){\bf 2}
\rput(6,6.5){\bf 3}
\rput(7.5,8){\bf 4}
\rput(-.6,4.4){\bf 5}
\rput(4.7,6.7){\bf 6}
\rput(7.3,9.3){\bf 7}
\rput(-.5,6){\bf 8}
\rput(4,8){\bf 9}
\rput(6.5,9.5){\bf 10}
\end{pspicture}
\caption{Dividing the region into 10 regions.}
  \label{fig:division2}
\end{figure}%
Therefore, to prove Proposition~\ref{prop:two} it remains to show the following
proposition.
\begin{prop}\label{prop:identity}
  For nonnegative integers $n_1,n_2,a,b$, and $n=n_1+n_2$, we have
\begin{multline*}
\sum_{i,j\ge0} q^{(n_1-a)i+(n_2-b)j-ij+i^2+j^2}
\qbinom{n_1}{a-i}\qbinom{n_2}{b-j} \qbinom{n_1+n_2}{n_1+i-j}
\qbinom{n_1+i}{n_1} \qbinom{n_2+j}{n_2}\\
= \qbinom{n}{n_1}\sum_{i\ge0} q^{(n-a-b)i + i^2}
\qbinom{n+i}i \qbinom{n}{a-i} \qbinom{n}{b-i}.
\end{multline*}
\end{prop}

\begin{proof}
  We will follow the standard notation in hypergeometric series, see
  \cite{Gasper2004}.  It is straightforward to check that the identity in this
  proposition is the $(a,b,x,y)\mapsto (q^{-a},q^{-b},q^{n_1},q^{n_2})$
  specialization of
\begin{multline}\label{eq:hyper}
\sum_{i,j\ge0} (-x)^i (-y)^j q^{\binom{i+1}{2}+\binom{j+1}{2}-ij}
\frac{(a,xq;q)_i (b,yq;q)_j} {(q,axq;q)_i(q,byq;q)_j(xq;q)_{i-j}(yq;q)_{j-i}} \\
=\frac{(xq,yq,axyq,bxyq;q)_{\infty}} {(xyq,xyq,axq,byq;q)_{\infty}}
\qhyper{3}{2}{a,b,xyq}{axyq,bxyq}{xyq}.
\end{multline}
We now prove \eqref{eq:hyper} as follows. Observe that the left hand side of
\eqref{eq:hyper} can be written as
\begin{align*}
&\sum_{i\ge0} \frac{(xyq)^i (a,1/y;q)_i}{(q,axq;q)_i} 
\qhyper{3}{2}{b,yq,q^{-i}/x}{byq,yq^{1-i}}{xyq}\\
=&\frac{(yq,bxyq;q)_{\infty}}{(xyq,byq;q)_{\infty}}
\sum_{i\ge0} \frac{(xyq)^i (a,1/y;q)_i}{(q,axq;q)_i} 
\qhyper{3}{2}{b,xyq,q^{-i}}{bxyq,yq^{1-i}}{yq}\\
=&\frac{(yq,bxyq;q)_{\infty}}{(xyq,byq;q)_{\infty}}
\sum_{i\ge0} \frac{(xyq)^i (a,1/y;q)_i}{(q,axq;q)_i} 
\sum_{j\ge0} \frac{(b,xyq,q^{-i};q)_j}{(q,bxyq,yq^{1-i};q)_j} (yq)^j,
\end{align*}
where we use \cite[Eq. (III.9)]{Gasper2004} with $(a,b,c,d,e)\mapsto
(b,yq,q^{-i}/x,yq^{1-i},byq)$. By replacing $i$ with $i+j$ and interchaning the
sums, we obtain that the above equals
\[ 
\frac{(yq,bxyq;q)_{\infty}}{(xyq,byq;q)_{\infty}}
\sum_{j\ge0} \frac{(xyq)^j (a,b,xyq;q)_j}{(q,axq,bxyq;q)_j}
\qhyper{2}{1}{aq^j,1/y}{axq^{j+1}}{xyq}.
\]
The $q$-Gauss sum \cite[Eq. (II.8)]{Gasper2004} completes the proof of
\eqref{eq:hyper}.
\end{proof}

\section{A bijection from Dyck tilings to matchings}
\label{sec:bijection}

In this section we find a bijection sending Dyck tilings to Hermite histories,
which are in simple bijection with complete matchings. We start by defining
these objects. 

A \emph{(complete) matching} on $[2n]$ is a set of pairs $(i,j)$ of integers in
$[2n]$ with $i<j$ such that each integer in $[2n]$ appears exactly once.  We
denote by $\M(2n)$ the set of matchings on $[2n]$.  It is convenient to
represent $\pi\in\M(2n)$ by the diagram obtained by joining $i$ and $j$
with an arc for each $(i,j)\in \pi$ as shown in Figure~\ref{fig:matching}.  We
define the \emph{shape} of $\pi$ to be the Dyck path such that the $i$th step is
an up step if $(i,j)\in \pi$ for some $j$, and a down step otherwise, see
Figure~\ref{fig:matching}. For a Dyck path $\mu$, the set of matchings with shape
$\mu$ is denoted by $\M(\mu)$.  A \emph{crossing} (resp.~\emph{nesting}) of
$\pi\in\M(2n)$ is a set of two pairs $(i,j)$ and $(i',j')$ in $\pi$ such that
$i<i'<j<j'$ (resp.~$i<i'<j'<j$).  The number of crossings (resp.~nestings) of
$\pi$ is denoted by $\cro(\pi)$ (resp.~$\nest(\pi)$). For example, if $\pi$ is
the matching in Figure~\ref{fig:matching}, we have $\cro(\pi)=2$ and
$\nest(\pi)=1$.

\begin{figure}
  \centering
  \begin{pspicture}(1,0)(14,4)
\psset{linewidth=1pt, dotsize=5pt, unit=20pt}
    \multido{\n=1+1}{8}{\vput{\n}}
    \edge15 \edge23 \edge47 \edge68
  \end{pspicture}
  \begin{pspicture}(1,0)(4,5)
  \end{pspicture}
  \begin{pspicture}(0,0)(4,4)
\dyckgrid4
\psline(0,0)(0,2)(1,2)(1,3)(2,3)(2,4)(4,4)
  \end{pspicture}
  \caption{The diagram (left) of the matching $\{(1,5), (2,3), (4,7), (6,8)\}$
    and its shape (right).}
  \label{fig:matching}
\end{figure}
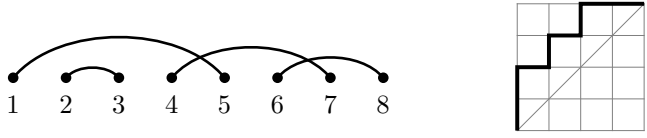

A \emph{Hermite history} of length $2n$ is a pair $(\mu,H)$ of a Dyck path
$\mu\in\Dyck(2n)$ and a labeling $H$ of the down steps of $\mu$ such that the
label of a down step of height $h$ is an integer in $\{0,1,\dots,h-1\}$. We
denote by $\HH(2n)$ the set of Hermite histories of length $2n$, and by
$\HH(\mu)$ the set of Hermite histories with Dyck path $\mu$.  There is a
well-known bijection $\zeta:\M(2n)\to\HH(2n)$, see \cite{ViennotOP} or
\cite{JVR}. For $\pi\in \M(2n)$, the corresponding Hermite history
$\zeta(\pi)=(\mu,H)$ is defined as follows. The Dyck path $\mu$ is the shape of
$\pi$. For a down step $D$ of $\mu$, if it is the $j$th step, there is a pair
$(i,j)\in \pi$. Then the label of $D$ is defined to be the number of pairs
$(i',j')\in \pi$ such that $i<i'<j<j'$. For example, if $\pi$ is the matching in
Figure~\ref{fig:matching}, then $\mu$ is the Dyck path in
Figure~\ref{fig:matching} and the labels of the downs steps are 0, 1, 1, and 0
in this order. Note that $\zeta$ is also a bijection from $\M(\mu)$ to
$\HH(\mu)$.

For $\mu\in \Dyck(2n)$ and $(\mu,H)\in\HH(2n)$, we define
\begin{align*}
  \HT(\mu) &= \sum_{c\in\C(\mu)} \left(\HT(c)-1\right),\\
  \|H\| &= \sum_{i\in H} i.
\end{align*}
The next lemma easily follows from the construction of the map $\zeta$.

\begin{lem}\label{lem:cr}
Let $\zeta(\pi)=(\mu,H)$. Then $\|H\| = \cro(\pi)$ and $\HT(\mu) =
  \cro(\pi)+\nest(\pi)$.
\end{lem}

From now on we will use the following notations: for $\lambda,\mu\in\Dyck(2n)$,
\begin{align*}
\D(\lambda/*) &= \bigcup_{\nu\in\Dyck(2n)} \D(\lambda/\nu),\\
\D(*/\mu) &= \bigcup_{\nu\in\Dyck(2n)} \D(\nu/\mu),\\
\D(2n) &= \bigcup_{\nu,\rho\in\Dyck(2n)} \D(\nu/\rho).  
\end{align*}

For a Dyck tile $\eta$, we define the \emph{entry} (resp.~\emph{exit}) of $\eta$
to be the north border (resp.~the south border) of the northeast cell (resp.~
the southwest cell) of $\eta$.

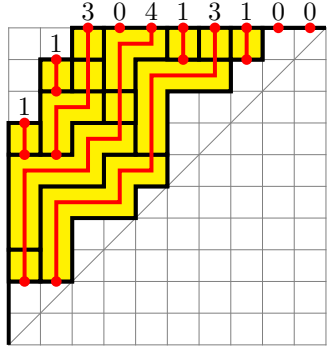
\begin{figure}
  \centering
  \begin{pspicture}(0,0)(11,11) 
\dyckgrid{10}
\psline(0,0)(0,7)(1,7)(1,9)(2,9)(2,10)(10,10)
\psline(0,2)(2,2)(2,4)(4,4)(4,5)(5,5)(5,8)(7,8)(7,9)(8,9)(8,10)(10,10)
\psset{fillstyle=solid,fillcolor=yellow}
\pspolygon(1,2)(1,5)(3,5)(3,6)(5,6)(5,5)(4,5)(4,4)(2,4)(2,2)
\rput(-1,1){\pspolygon(1,2)(1,5)(3,5)(3,6)(5,6)(5,5)(4,5)(4,4)(2,4)(2,2)}
\pspolygon(1,6)(1,8)(3,8)(3,7)(2,7)(2,6)
\rput(2,2){\pspolygon(1,6)(1,8)(3,8)(3,7)(2,7)(2,6)}
\pspolygon(4,6)(4,9)(7,9)(7,8)(5,8)(5,6)
\psframe(0,2)(1,3)
\rput(0,4){\psframe(0,2)(1,3)}
\rput(1,6){\psframe(0,2)(1,3)}
\rput(2,6){\psframe(0,2)(1,3)}
\rput(2,7){\psframe(0,2)(1,3)}
\rput(3,5){\psframe(0,2)(1,3)}
\rput(5,7){\psframe(0,2)(1,3)}
\rput(6,7){\psframe(0,2)(1,3)}
\rput(7,7){\psframe(0,2)(1,3)}
\psset{linecolor=red,fillstyle=none}
\psline(.5,7)(.5,6)
\psline(1.5,9)(1.5,8)
\psline(2.5,10)(2.5,7.5)(1.5,7.5)(1.5,6)
\psline(4.5,10)(4.5,9.5)(3.5,9.5)(3.5,6.5)(2.5,6.5)(2.5,5.5)(.5,5.5)(.5,2)
\psline(5.5,10)(5.5,9)
\psline(6.5,10)(6.5,8.5)(4.5,8.5)(4.5,5.5)(3.5,5.5)(3.5,4.5)(1.5,4.5)(1.5,2)
\psline(7.5,10)(7.5,9)
\psset{dotsize=4pt}
\psdot(.5,7) \rput(.5,7.5){$1$}
\psdot(1.5,9) \rput(1.5,9.5){$1$}
\psdot(2.5,10) \rput(2.5,10.5){$3$}
\psdot(3.5,10) \rput(3.5,10.5){$0$}
\psdot(4.5,10) \rput(4.5,10.5){$4$}
\psdot(5.5,10) \rput(5.5,10.5){$1$}
\psdot(6.5,10) \rput(6.5,10.5){$3$}
\psdot(7.5,10) \rput(7.5,10.5){$1$}
\psdot(8.5,10) \rput(8.5,10.5){$0$}
\psdot(9.5,10) \rput(9.5,10.5){$0$}
\psdot(.5,6) \psdot(.5,2)
\psdot(1.5,8) \psdot(1.5,6) \psdot(1.5,2)
\psdot(5.5,9) \psdot(7.5,9)
\end{pspicture}%
 \caption{An example of the map $\psi$.}
  \label{fig:psi}
\end{figure}

For $T\in\D(*/\mu)$, we define $\psi(T)=(\mu,H)$ as follows. The label of a down
step $s$ of $\mu$ is the number of Dyck tiles that we pass in the following
process.  We start from $s$ and travel to the south until we reach a border that
is not an entry; if we arrive at the entry of a Dyck tile, then continue traveling
from the exit of the Dyck tile, see Figure~\ref{fig:psi}. Observe that every
tile is traveled exactly once, which can be checked using the definition of a
truncated Dyck tiling. Thus we have $|T|=\|H\|$.  It is easy to see that the map
$\psi$ has the same recursive structure as the proof of Theorem~\ref{thm:2} in
Section~\ref{sec:proof2}. Thus $\psi:\D(*/\mu)\to\HH(\mu)$ is a bijection. It is
also possible to construct the inverse map of $\psi$, but it is more complicated
than $\psi$.

\begin{thm}\label{thm:bijection}
  Given a Dyck path $\mu\in\Dyck(2n)$, the map $\psi$ gives a bijection $\psi:
  \D(*/\mu) \to \HH(\mu)$ such that if $\psi(T)=(\mu,H)$, then
  $|T|=\|H\|$. Thus, $\zeta^{-1}\circ\psi : \D(*/\mu)\to\M(\mu)$ is a bijection
  such that if $(\zeta^{-1}\circ\psi)(T) = \pi$, then $|T|=\cro(\pi)$. 
\end{thm}

We now discuss several applications of Theorem~\ref{thm:bijection}.  First of all,
since
\begin{equation}
  \label{eq:21}
\sum_{H\in \HH(\mu)} q^{\|H\|} = \prod_{c\in \C(\mu)} \qint{\HT(c)},
\end{equation}
Theorem~\ref{thm:bijection} gives a bijective proof of Theorem~\ref{thm:2}.

By Theorem~\ref{thm:bijection}, $\D(2n)$, $\HH(2n)$, and $\M(2n)$ have the same
cardinality $(2n-1)!!=1\cdot3\cdots(2n-1)$. Therefore, we have
\[
|\D(2n)| = (2n-1)!!,
\]
which was also conjectured by Kenyon and Wilson (private communication with
David Wilson).

For $T\in\D(\lm)$, we define $\HT(T)=\HT(\mu)$.

\begin{cor}\label{cor:pq}
  We have
\[
\sum_{T\in \D(2n)} p^{\HT(T)-|T|} q^{|T|} 
= \sum_{\pi\in\M(2n)} p^{\nest(\pi)} q^{\cro(\pi)}.
\]
\end{cor}
\begin{proof}
By Lemma~\ref{lem:cr} and Theorem~\ref{thm:bijection}, we have
  \begin{align*}
\sum_{T\in \D(2n)} p^{\HT(T)-|T|} q^{|T|} 
 &= \sum_{\mu\in\Dyck(2n)} \sum_{T\in\D(*/\mu)} p^{\HT(\mu)-|T|} q^{|T|}  \\
 &= \sum_{\mu\in\Dyck(2n)} \sum_{\pi\in\M(\mu)} p^{\nest(\pi)} q^{\cro(\pi)}  \\
 &= \sum_{\pi\in\M(2n)} p^{\nest(\pi)} q^{\cro(\pi)}.
  \end{align*}
\end{proof}

It is known that the two statistics $\cro$ and $\nest$ have joint symmetric
distribution over matchings, see \cite[Corollary~1.4]{Klazar2006} or
\cite[(1.7)]{Kasraoui2006}. In other words,
\[
\sum_{\pi\in\M(2n)} p^{\nest(\pi)} q^{\cro(\pi)} = \sum_{\pi\in\M(2n)} p^{\cro(\pi)} q^{\nest(\pi)}.
\]
Thus, by Corollary~\ref{cor:pq} we get the following non-trivial identity:
\[
\sum_{T\in \D(2n)} p^{|T|} q^{\HT(T)-|T|}= \sum_{T\in \D(2n)} p^{\HT(T)-|T|} q^{|T|}.
\]

Let $D_n(p,q)$ be the sum in Corollary~\ref{cor:pq}:
\[
D_n(p,q) = \sum_{T\in \D(2n)} p^{\HT(T)-|T|} q^{|T|}.
\]
By Flajolet's theory on continued fractions \cite{Flajolet1982}, we have
\[
\sum_{n\ge0} D_n(p,q) x^n =
\cfrac{1}{
  1 - \cfrac{\pqint 1 x}{
    1 - \cfrac{\pqint2 x}{
      1 - \cdots
    }
  }
},
\]
where $\pqint n = p^{n-1} + p^{n-2}q + \cdots + pq^{n-2} + q^{n-1}$. By
Viennot's theory \cite{ViennotOPF, ViennotOP}, $D_n(p,q)$ is equal to the $2n$th
moment of the orthogonal polynomial $H_n(x;p,q)$ defined by $H_{-1}(x;p,q)=0$,
$H_{0}(x;p,q)=1$, and the three term recurrence
\[
H_{n+1}(x;p,q) = x H_n(x;p,q) - \pqint{n} H_{n-1}(x;p,q).
\]
In particular, $H_n(x;1,q)$ is the continuous $q$-Hermite polynomial and
$H_n(x;q,q^2)$ is the discrete $q$-Hermite polynomial, see \cite{IsmailStanton,
  SimionStanton}. There are known formulas for the $2n$th moments of
$H_n(x;1,q)$ and $H_n(x;q,q^2)$.  For the $2n$th moment of $H_n(x;1,q)$, we
have the Touchard-Riordan formula which has various proofs, see
\cite{Cigler2011, JVKim, JVR, Penaud1995, Prodinger, Riordan1975, Touchard1952}:
\begin{equation}
 \label{eq:22}
\sum_{\pi\in \M(2n)} q^{\cro(\pi)} = \frac{1}{(1-q)^n} \sum_{k=0}^n 
\left(\binom{2n}{n-k}-\binom{2n}{n-k-1}\right)(-1)^k q^{\binom{k+1}2}.  
\end{equation}
For the $2n$th moment of $H_n(x;q,q^2)$, we have the following formula, see \cite[Proof
of Corollary 2]{IsmailStanton} or \cite[(5.4)]{SimionStanton}:
\begin{equation}
 \label{eq:23}
\sum_{\pi\in \M(2n)} q^{2\cro(\pi)+\nest(\pi)} = \qint{2n-1}!!,
\end{equation}
where $\qint{2n-1}!!=\qint 1 \qint 3 \cdots \qint{2n-1}$.

By Corollary~\ref{cor:pq}, \eqref{eq:22} and \eqref{eq:23} we obtain the
following corollary.
\begin{cor}\label{cor:q1}
We have
\[
\sum_{T\in\D(2n)} q^{|T|} = \frac{1}{(1-q)^n} \sum_{k=0}^n 
\left(\binom{2n}{n-k}-\binom{2n}{n-k-1}\right)(-1)^k q^{\binom{k+1}2},
\]
\[
\sum_{T\in\D(2n)} q^{\HT(T)+|T|} = \qint{2n-1}!!.
\]
\end{cor}

\section{Final remarks}
\label{sec:final-remarks}

We can generalize the matrix $M$ in the introduction as follows.  The matrix
$M(p,q)$ is defined by
\[
M(p,q)_{\lambda,\mu} = 
\left\{ 
  \begin{array}{ll}
    p^{|\lm|} q^{d(\lambda,\mu)}, & \mbox{if $\lambda\succ\mu$;}\\
0, & \mbox{otherwise,}
 \end{array}
\right.
\]
where $d(\lambda,\mu)$ is the number of reversed matching pairs when going from
$\mu$ to $\lambda$.  Then $M=M(1,1)$.  Recall that Kenyon and Wilson
\cite[Theorem~1.5]{Kenyon2011} proved that
\[
  M^{-1}_{\lambda,\mu}  = (-1)^{|\lm|} \times |\D(\lm)|.  
\]
It is not hard to see that the proof of the above identity in \cite{Kenyon2011}
also implies the following identity, which was first observed by Matja\v z
Konvalinka (personal communication with Matja\v z Konvalinka):
\begin{equation}
  \label{eq:Mpq}
M(p,q)^{-1}_{\lambda,\mu} = \sum_{T\in\D(\lm)} (-p)^{|\lm|} q^{|T|}.  
\end{equation}

Note that Theorem~\ref{thm:1} (resp.~Theorem~\ref{thm:2}) is a formula for the
sum of the absolute values of the entries in a row of $M(q^{1/2},q^{1/2})$
(resp.~a column of $M(1,q)$). Such a sum using $M(p,q)$ does not factor nicely,
so it seems more difficult to find a formula for the sum.

In Section~\ref{sec:bijection} we have found a bijection
$\psi:\D(*/\mu)\to\HH(\mu)$ which gives a bijective proof of
Theorem~\ref{thm:2}. A bijective proof of Theorem~\ref{thm:1} is given
in \cite{KMPW}.

Finally we note that, although it is not directly related to this paper, Dyck
tiles are also used in \cite{Brenti2002} as a combinatorial tool for
Kazhdan-Lusztig polynomials.

\section*{Acknowledgement}
I would like to thank Dennis Stanton for drawing my attention to Kenyon and
Wilson's conjectures, and for helpful discussion.  I also thank Matja\v z
Konvalinka, David Wilson, and Victor Reiner for helpful comments, and Ole
Warnaar for providing me the proof of Proposition~\ref{prop:identity} which has
replaced the previous proof.

% \bibliographystyle{abbrv}
% \bibliography{/Users/jskim/bongs/math/research/bib/ref}

\end{document}